\documentclass[a4paper,12pt,twoside]{amsart}
\usepackage{vmargin}
\usepackage{amssymb}
\usepackage[all]{xy}
\usepackage{amsmath}
\usepackage{mathrsfs}
\usepackage{bbold}
\usepackage[dvips]{graphicx}
\usepackage{vmargin}
\usepackage{color}
\usepackage{amscd}
\DeclareGraphicsExtensions{.pdf, .png, .jpg, .mps}
\usepackage[colorlinks=true,pdfpagemode=FullScreen]{hyperref}
\definecolor{myc}{cmyk}{0.0009,0.8,0.8,0.00}

\RequirePackage[capitalize,nameinlink]{cleveref}[0.19]

\crefname{section}{section}{sections}
\crefname{subsection}{subsection}{subsections}
\Crefname{section}{Section}{Sections}
\Crefname{subsection}{Subsection}{Subsections}

\Crefname{figure}{Figure}{Figures}

\crefformat{equation}{\textup{#2(#1)#3}}
\crefrangeformat{equation}{\textup{#3(#1)#4--#5(#2)#6}}
\crefmultiformat{equation}{\textup{#2(#1)#3}}{ and \textup{#2(#1)#3}}
{, \textup{#2(#1)#3}}{, and \textup{#2(#1)#3}}
\crefrangemultiformat{equation}{\textup{#3(#1)#4--#5(#2)#6}}%
{ and \textup{#3(#1)#4--#5(#2)#6}}{, \textup{#3(#1)#4--#5(#2)#6}}{, and \textup{#3(#1)#4--#5(#2)#6}}

\Crefformat{equation}{#2Equation~\textup{(#1)}#3}
\Crefrangeformat{equation}{Equations~\textup{#3(#1)#4--#5(#2)#6}}
\Crefmultiformat{equation}{Equations~\textup{#2(#1)#3}}{ and \textup{#2(#1)#3}}
{, \textup{#2(#1)#3}}{, and \textup{#2(#1)#3}}
\Crefrangemultiformat{equation}{Equations~\textup{#3(#1)#4--#5(#2)#6}}%
{ and \textup{#3(#1)#4--#5(#2)#6}}{, \textup{#3(#1)#4--#5(#2)#6}}{, and \textup{#3(#1)#4--#5(#2)#6}}

\crefdefaultlabelformat{#2\textup{#1}#3}

 \numberwithin{equation}{section}

\newcommand{\N}{{\mathbb N}}
\newcommand{\A}{{\mathcal A}}

\newcommand{\R}{{\mathbb R}}
\newcommand{\h}{{\mathcal H}}
\newcommand{\Q}{{\mathbb Q}}

\newcommand{\C}{{\mathbb C}}

\newcommand{\E}{{\mathcal E}}
\newcommand{\K}{{\mathscr K}}
\newcommand{\M}{{\mathcal M}}
\newcommand{\G}{{\mathcal G}}
\newcommand{\dif}{\mathop{}\!\mathrm{d}}

\def\poscalr#1#2{\langle#1,#2\rangle}

\def\p{\partial}

\def\Op{\text{\rm Op}}

\def\id{{\rm Id}}

\newtheorem{Thm}{Theorem}[section]
\newtheorem{thm}[Thm]{Theorem}
\newtheorem{coro}[Thm]{Corollary}
\newtheorem{rem}[Thm]{Remark}
\newtheorem{lem}[Thm]{Lemma}
\newtheorem{prop}[Thm]{Proposition}
\newtheorem{defi}[Thm]{Definition}

\begin{document}
\author{Jingrui NIU}
\thanks{Universit\'e Paris-Saclay, Laboratoire de math\'ematiques d'Orsay, UMR 8628 du CNRS, B\^atiment 307, 91405 Orsay Cedex, France.
email: \texttt{jingrui.niu@universite-paris-saclay.fr}.}

\title[Simultaneous Control]{Simultaneous Control of Wave Systems}
\begin{abstract}
In this paper, we study the simultaneous controllability of wave systems with different speeds in an open domain of $\R^d$, $d\in\N^*$, and under a uniqueness assumption for eigenfunctions, we prove the exact controllability with a single control command. We then study this uniqueness assumption and both provide a counterexample (for which we hence only obtain a partial controllability result on a co-finite dimensional space) and examples to ensure the unique continuation property. For the case of constant coefficients and possibly multiple control functions, we prove the controllability property is equivalent to an appropriate Kalman rank condition.
\end{abstract}

\date{\today}
\maketitle
{\footnotesize
\baselineskip=0.72
\normalbaselineskip
\tableofcontents}

\section{Introduction}
Let $\Omega\subset\R^d$, $d\in\N^*$, be a bounded, and smooth domain. For positive constants $\alpha$ and $\beta$, let $k_{ij}(x):\Omega\rightarrow\R$, $1\leq i,j\leq d$ be smooth functions which satisfy:
\begin{equation}
    k_{ij}(x)=k_{ji}(x),\alpha|\xi|^2\leq\sum_{1\leq i,j\leq d}k_{ij}(x)\xi_i\xi_j\leq \beta|\xi|^2,\forall x\in\Omega,\forall \xi\in\R^d.
\end{equation}
Define $K(x)$ to be the symmetric positive definite matrix of coefficients $k_{ij}(x)$. Moreover, we define the density function $\kappa(x)=\frac{1}{\sqrt{\det(K(x))}}$. 
We also define the Laplacian by $\Delta_K=\frac{1}{\kappa(x)}div(\kappa(x)K\nabla\cdot)$ on $\Omega$ and the d'Alembert operator $\Box_K=\p_t^2-\Delta_K$ on $\R_t\times\Omega$. 
We assume that $\omega$ is a nonempty open subset of $\Omega$. We consider the interior simultaneous controllability problem for the following wave system:
\begin{equation}\label{eq:control system of different metrics}
\left\{
  \begin{array}{l}
    \Box_{K_1}u_1=b_1f\mathbf{1}_{]0,T[}(t)\mathbf{1}_{\omega}(x)\text{ in }]0,T[\times\Omega,\\
    \Box_{K_2}u_2=b_2f\mathbf{1}_{]0,T[}(t)\mathbf{1}_{\omega}(x)\text{ in }]0,T[\times\Omega,\\
    \vdots\\
    \Box_{K_n}u_n=b_nf\mathbf{1}_{]0,T[}(t)\mathbf{1}_{\omega}(x)\text{ in }]0,T[\times\Omega,\\
    u_j=0\quad\text{ on }]0,T[\times\p\Omega,1\leq j\leq n,\\
    u_j(0,x)=u_j^0(x),\quad \p_tu_j(0,x)=u_j^1(x), 1\leq j\leq n.
  \end{array}
\right.
\end{equation}
Here, we choose $K_i$($1\leq i\leq n$) to be $n$ different symmetric positive definite matrices. The state of the system is $(u_1,\p_t u_1,\cdots,u_n,\p_t u_n)$ and $f$ is our control function. $b_i$ are $n$ nonzero constant coefficients. In this paper, we mainly consider the exact controllability for the system \cref{eq:control system of different metrics} given by the following definition. 
\begin{defi}[Exact Controllability]
We say that the system \cref{eq:control system of different metrics} is exactly controllable if for any initial data $(u_1^0,u_1^1,\cdots,u_n^0,u_n^1)\in (H^1_0(\Omega)\times L^2(\Omega))^n$ and any target data $(U_1^0,U_1^1,\cdots,U_n^0,U_n^1)\in (H^1_0(\Omega)\times L^2(\Omega))^n$, there exists a control function $f\in L^2(]0,T[\times\omega)$ such that the solution of the system \cref{eq:control system of different metrics} with initial data $(u_1,\p_t u_1,\cdots,u_n,\p_t u_n)|_{t=0}=(u_1^0,\cdots,u_n^1)$ satisfies $(u_1,\p_t u_1,\cdots,u_n,\p_t u_n)|_{t=T}\\=(U_1^0,\cdots,U_n^1)$.
\end{defi}
Moreover, we also consider the partial exact controllability for the system \cref{eq:control system of different metrics} given by the following definition.
\begin{defi}
Let $\Pi$ be a projection operator of $(H^1_0(\Omega)\times L^2(\Omega))^n$.
We say that the system \cref{eq:control system of different metrics} is $\Pi-$exactly controllable if for any initial data $(u_1^0,u_1^1,\cdots,u_n^0,u_n^1)\in (H^1_0(\Omega)\times L^2(\Omega))^n$ and any target data $(U_1^0,U_1^1,\cdots,U_n^0,U_n^1)\in (H^1_0(\Omega)\times L^2(\Omega))^n$, there exists a control function $f\in L^2(]0,T[\times\omega)$ such that the solution of \cref{eq:control system of different metrics} with initial data $(u_1,\p_t u_1,\cdots,u_n,\p_t u_n)|_{t=0}=(u_1^0,u_1^1,\cdots,u_n^0,u_n^1)$ satisfies
\begin{equation*}
    \Pi(u_1,\p_t u_1,\cdots,u_n,\p_t u_n)|_{t=T}=\Pi(U_1^0,U_1^1,\cdots,U_n^0,U_n^1).
\end{equation*} 
If we only impose that $\Pi(u_1,\p_t u_1,\cdots,u_n,\p_t u_n)|_{t=T}=0$, we say that the system \cref{eq:control system of different metrics} is $\Pi-$null controllable.
\end{defi}
\begin{prop}\label{thm:equivalence thm}
For the system \cref{eq:control system of different metrics}, the $\Pi-$null controllability is equivalent to the $\Pi-$exact controllability. 
\end{prop}
\begin{proof}
We follow closely the proof of  \cite[Theorem 2.41]{coron2007control}. It is clear that ($\Pi-$exact controllability) $\Longrightarrow$ ($\Pi-$null controllability). So we focus on the proof of the converse. We define the operator
\begin{equation}
    \mathscr{A}=\left(
            \begin{array}{cccc}
              0 & -1 & \cdots& 0 \\
              -\Delta_{K_1}  & 0 & \cdots & 0 \\
              \vdots & \vdots & 0 & -1 \\
              0 & 0 &  -\Delta_{K_n} & 0 \\
            \end{array}
          \right)  .
\end{equation}
the system \cref{eq:control system of different metrics} is equivalent to 
\begin{equation}\label{eq:ODE-system}
    \p_t y=-\mathscr{A}y+\Tilde{B}f\mathbf{1}_{]0,T[}(t)\mathbf{1}_{\omega}(x), y|_{t=0}=y(0),
\end{equation}
where 
\begin{equation*}
    y=\left(
\begin{array}{c}
     u_1 \\
     \p_t u_1\\
     \vdots\\
     u_n\\
     \p_t u_n
\end{array}
\right),\quad 
y(0)=\left(
\begin{array}{c}
     u^0_1 \\
     u^1_1\\
     \vdots\\
     u^0_n\\
     u^1_n
\end{array}
\right)\quad
\text{and }
\Tilde{B}=\left(
\begin{array}{c}
     0 \\
     b_1\\
     \vdots\\
     0\\
     b_n
\end{array}
\right).
\end{equation*}
Let us consider $S(t)$ the semi-group generated by $\mathscr{A}$. Let $y^0\in (H^1_0(\Omega)\times L^2(\Omega))^n$ and $y^1\in(H^1_0(\Omega)\times L^2(\Omega))^n$. Since the system \cref{eq:control system of different metrics} is $\Pi-$null controllable, we obtain that there exists $f$ such that the solution $\Tilde{y}$ of the Cauchy problem 
\begin{equation}
    \p_t\Tilde{y}=-\mathscr{A}\Tilde{y}+\Tilde{B}f\mathbf{1}_{]0,T[}(t)\mathbf{1}_{\omega}(x), y|_{t=0}=y^0-S(-T)y^1
\end{equation}
satisfies $\Pi\Tilde{y}(T)=0$. For the Cauchy problem
\begin{equation}
    \p_t y=-\mathscr{A}y+\Tilde{B}f\mathbf{1}_{]0,T[}(t)\mathbf{1}_{\omega}(x), y|_{t=0}=y^0,
\end{equation}
the solution $y$ is given by
\begin{equation}
    y(t)=\Tilde{y}(t)+S(t-T)y^1, \forall t\in [0,T].
\end{equation}
Hence, we obtain that $y(T)=\Tilde{y}(T)+y^1$. In particular, we know that $\Pi y(T)=\Pi y^1$ since $\Pi\Tilde{y}(T)=0$. We now obtain the $\Pi-$exact controllability for the system \cref{eq:control system of different metrics}. 
\end{proof}

According to the Hilbert Uniqueness Method of J.-L. Lions \cite{HUM}, the controllability property is equivalent to an observability inequality for the adjoint system. In particular, when we focus on our system \cref{eq:control system of different metrics}, the exact controllability is equivalent to proving the following observability inequality: $\exists C>0$ such that for any solution of the adjoint system:
\begin{equation}\label{eq:adjoint system of different metrics}
\left\{
  \begin{array}{l}
    \Box_{K_1}v_1=0\text{ in }]0,T[\times\Omega,\\
    \Box_{K_2}v_2=0\text{ in }]0,T[\times\Omega,\\
    \vdots\\
    \Box_{K_n}v_n=0\text{ in }]0,T[\times\Omega,\\
    v_j=0\quad\text{ on }]0,T[\times\p\Omega,1\leq j\leq n,\\
    v_j(0,x)=v_j^0(x),\quad \p_t v_j(0,x)=v_j^1(x), 1\leq j\leq n,
  \end{array}
\right.
\end{equation}
we have
\begin{equation}\label{eq:general ob}
C\int_0^T\int_{\omega}{|b_1\kappa_1v_1+\cdots+b_n\kappa_nv_n|}^2dxdt\geq \sum_{i=1}^n(||v_i^0||^2_{L^2}+||v^1_i||^2_{H^{-1}}).
\end{equation}
For the partial controllability, we have a similar result. The $\Pi-$exact controllability of the system \cref{eq:control system of different metrics} is equivalent to proving the following observability inequality: $\exists C>0$ such that for any solution of the adjoint system:
\begin{equation}\label{eq:P-adjoint system of different metrics}
\left\{
  \begin{array}{l}
    \Box_{K_1}v_1=0\text{ in }]0,T[\times\Omega,\\
    \Box_{K_2}v_2=0\text{ in }]0,T[\times\Omega,\\
    \vdots\\
    \Box_{K_n}v_n=0\text{ in }]0,T[\times\Omega,\\
    v_j=0\quad\text{ on }]0,T[\times\p\Omega,1\leq j\leq n,\\
    (v_1(0,x),\p_t v_1(0,x),\cdots,v_n(0,x) \p_t v_n(0,x))=\Pi^* V^0, 
  \end{array}
\right.
\end{equation}
where $V^0\in (L^2\times H^{-1})^n$ and $\Pi^*$ is the adjoint operator of the projector $\Pi$, we have 
\begin{equation}\label{eq:P-general ob}
C\int_0^T\int_{\omega}{|b_1\kappa_1v_1+\cdots+b_n\kappa_nv_n|}^2dxdt\geq ||\Pi^* V^0||^2_{(L^2\times H^{-1})^n}. 
\end{equation}
This is an easy consequence of \cref{thm:equivalence thm}, the conservation of energy for system \cref{eq:control system of different metrics} and \cite[Chapter 4, Proposition 2.1]{khodja2017partial}.

In order to study the observability inequality, a classical method is to  follow the abstract three-step process initialized by Rauch and Taylor \cite{3-step}(see also \cite{GCC}). It can be detailed as follows: 
\begin{itemize}
    \item Firstly, get the microlocal information on the observable region. Argue by contradiction to obtain different kinds of convergence in subdomain $]0,T[\times\omega$ and the whole domain $]0,T[\times\Omega$.
    \item Secondly, use microlocal defect measure (which is due to G\'erard \cite{defect_measure} and Tartar \cite{H-measure}), or propagation of singulaties theorem (see \cite{hormander1985analysis} Section 18.1) to prove a weak observability estimate:
    \begin{equation*}
\begin{aligned}
    \sum_{i=1}^n(||v_i^0||^2_{L^2}+||v^1_i||^2_{H^{-1}})&\\\leq C(\int_0^T\int_{\omega}|\sum_{j=1}^n b_j\kappa_jv_j|^2dxdt
    &+\sum_{i=1}^n(||v_i^0||^2_{H^{-1}}+||v^1_i||^2_{H^{-2}})).
\end{aligned}
\end{equation*}
    \item Thirdly, use unique continuation properties of eigenfunctions to obtain the original observability inequality \cref{eq:general ob}.
\end{itemize}
  For the high frequency estimates, a very natural condition is to assume that the control set satisfies the Geometric Control Condition(GCC).
\begin{defi}
For $\omega\subset \Omega$ and $T>0$, we shall say that the pair $(\omega, T,p_K)$ satisfies GCC if every general bicharacteristic of $p_K$ meets $\omega$ in a time $t<T$, where $p_K$ is the principal symbol of $\Box_K$.
\end{defi}
We will give the definition of bicharacteristics in \Cref{sec:Geometric Preliminaries}. This condition was raised by Bardos, Lebeau, and Rauch \cite{bardos1988controle} when they considered the controllability of a scalar wave equation and has now become a basic assumption for the controllability of wave equations. In \cite{burq1997condition}, the authors show that the geometric control condition is a necessary and sufficient condition for the exact controllability of the wave equation with Dirichlet boundary conditions and continuous boundary control functions. In order to study the low frequencies, we need to introduce the notion of unique continuation of eigenfunctions.
\begin{defi}
We say the system \cref{eq:control system of different metrics} satisfies the unique continuation of eigenfunctions if the following property holds: $\forall \lambda\in\C$, the only solution $(\phi_1,\cdots, \phi_n)\in (H^1_0(\Omega))^n$ of
\begin{equation*}
\left\{
  \begin{array}{l}
    -\Delta_{K_1}\phi_1=\lambda^2\phi_1\text{ in }\Omega,\\
    -\Delta_{K_2}\phi_2=\lambda^2\phi_2\text{ in }\Omega,\\
\cdots\\
 -\Delta_{K_n}\phi_n=\lambda^2\phi_n\text{ in }\Omega,\\
    b_1\kappa_1\phi_1+\cdots+b_n\kappa_n\phi_n=0\text{ in } \omega,\\
  \end{array}
\right.
\end{equation*}
is the zero solution $(\phi_1,\cdots, \phi_n)\equiv0$.
\end{defi}
There is a large literature on the controllability and observability of the wave equations. Several techniques have been applied to derive observability inequalities in various situations. This paper is mainly devoted to multi-speed wave systems coupled by the control functions only. For other interesting situations, we list some of the existing results and references:
\begin{itemize}
    \item For single wave equation, it is by now well-known that Bardos, Lebeau, and Rauch \cite{GCC} use microlocal analysis to prove the \cref{eq:general ob}-type observability inequality for a scalar wave equation. Other approaches for proving it can also be found in the literature, for example, using multipliers \cite{lions1988controlabilite,lagnese1983control}, using Carleman estimates \cite{haraux1991completion,baudouin2013global}, or completely constructive proof \cite{laurent2016uniform}, etc. 
    \item Although we now have a better picture on the controllabilty of a single wave equation, the controllability of systems of wave equations is still not totally understood. To our knowledge, most of the references concern the case of systems with the same principal symbol. Alabau-Boussouira and L\'eautaud \cite{2-coupled-wave} studied the indirect controllability of two coupled wave equations, in which their controllability result was established using a multi-level energy method introduced in \cite{alabau2003two}, and also used in \cite{alabau2013hierarchic,MR3163486}. Liard and Lissy \cite{liard2017kalman}, Lissy and Zuazua \cite{lissy2019internal} studied the observability and controllability of the coupled wave systems under the Kalman type rank condition. Moreover, we can find other controllability results for coupled wave systems, for example, Cui, Laurent, and Wang \cite{cui2018observability} studied the observability of wave equations coupled by first or zero order terms on a compact manifold. The microlocal defect measure when dealing with the single wave equation can also be extended to a system case. One can refer to Burq and Lebeau for the microlocal defect measure for systems \cite{measure-for-system}. 
    \item As for multi-speed case, Dehman, Le Roussau, and L\'eautaud considered two coupled wave equations with multi-speeds in \cite{2-multi-speed}. More related work is given by Tebou \cite{Tebou}, in which the author considered the simultaneous controllability of constant multi-speed wave system and derived some result in a semilinear setting in \cite{tebou-semilinear}.
\end{itemize}

\subsection{Plan of the paper}
The paper is organized as follows. Our main results are in
\Cref{sec:main} and \Cref{sec:Geometric Preliminaries} is devoted to introducing some geometric preliminaries. We include the descriptions of the boundary points, and give the precise definition of general bicharacteristics and the order of tangential contact with the boundary. 

In \Cref{sec:High Frequency Estimates}, we focus on the high frequency estimates. \Cref{sec:Microlocal defect measure} is devoted to introducing the microlocal defect measure and its basic properties, which is also the main tool for our proof. \Cref{sec:Proof of  thm:except finte dim} deals with the partial controllability,  and \Cref{sec:The proof of thm:thm of varying metric} is aimed to recover the exact controllability result in the whole energy space of initial conditions with the help of the unique continuation properties of eigenfunctions. In these two sections, we prove the \cref{thm:except finte dim}, and \cref{thm:thm of varying metric} respectively.

In \Cref{sec:Unique continuation of eigenfunctions}, we plan to deal with low frequency estimates, mainly discussing about the unique continuation properties of eigenfunctions. \Cref{sec:A counterexample} provides a counterexample to show that only assuming the hypotheses in \cref{thm:except finte dim} cannot ensure the unique continuation properties of eigenfunctions. Then, we add some stronger assumptions to obtain the unique continuation property. The first attempt is to require an analyticity condition, which is the example in \cref{thm:prop 2-eq}. The other attempt is to require constant coefficients in \Cref{sec:Constant Coefficient Case} and \Cref{sec:Proof of thm:constant result}, which is stated in \cref{thm:constant result}. \Cref{sec:Two Generic Properties} is about generic properties of metrics which ensure the unique continuation in dimension $1$ and $2$.

In \Cref{sec:Constant Coefficient Case with Multiple Control Functions}, we deal with the constant coefficient case with multiple control functions. We also discuss the corresponding Kalman rank condition in this setting.

In \Cref{appendix}, we include the proof of the equivalent condition of the Kalman rank condition in the case of multiple control functions.
\subsection{Ideas of the proof}  In our paper, we prove the controllability result by applying the Hilbert uniqueness method to prove the observability inequality of the adjoint system. In order to study the observability inequality, we always use an argument by contradiction. First, we try to prove a weak observability inequality by adding some low frequency part. To obtain the original observability inequality, we need to analyse the invisible solutions in the subdomain $\omega\times ]0,T[$ by proving the unique continuation properties of eigenfunctions. In section 4, we discuss some generic properties. We follow the ideas given by Uhlenbeck \cite{uhlenbeck1976generic}, using the transversality theorem to obtain generic properties.  

\section{Main results}
\label{sec:main}

In this paper, we mainly study the exact controllability for the system \cref{eq:control system of different metrics} and discuss the optimality of the given conditions. On the other hand, when we consider the constant coefficient case, we associate the controllability with the Kalman rank condition. Instead of considering the exact controllability, we can only consider the high frequency estimates to obtain a partial result. One can also see similar finite codimensional controllability results, for instance,  in \cite{cui2018observability} and \cite{codimensional}.
\begin{thm}\label{thm:except finte dim}
Given $T>0$, suppose that:
\begin{enumerate}
  \item $(\omega,T,p_{K_i})$ satisfies GCC, $i=1,2,\cdots,n$,
  \item\label{A3}$K_1>K_2>\cdots>K_n$ in $\omega$,
  \item $\Omega$ has no infinite order of tangential contact on the boundary.
\end{enumerate}
Then, there exists a finite dimensional subspace $E\subset (H^1_0(\Omega)\times L^2(\Omega))^n$ such that the system \cref{eq:control system of different metrics} is $\mathbb{P}-$exactly controllable, where $\mathbb{P}$ is the orthogonal projector on $E^\perp$.
\end{thm}
We will explain the concept of the order of contact in the \cref{sec:Geometric Preliminaries}.
\begin{rem}
We say that $K_1>K_2$ in $\omega$ if and only if $\forall x\in\omega$, $\forall\xi\in\R^d$ and $\xi\neq0$, $(\xi,K_1(x)\xi)>(\xi,K_2(x)\xi)$, where $(\cdot,\cdot)$ denotes the inner product of $\R^d$.
\end{rem}
\begin{rem}
The Assumption \eqref{A3} can be generalized as follows: let $\sigma$ be a permutation of $\{1,2,\cdots,n\}$, $K_{\sigma(1)}>K_{\sigma(2)}>\cdots>K_{\sigma(n)}$ in $\omega$.
\end{rem}
\begin{rem}
The same result holds for the laplacian operator
\begin{equation*}
    \Delta_{K,\kappa}=\frac{1}{\kappa(x)}div(\kappa(x)K(x)\nabla\cdot),
\end{equation*}
where we only assume that  $\kappa\in C^{\infty}(\Omega)$ without the restriction $\kappa(x)=\frac{1}{\sqrt{\det(K(x))}}$.
\end{rem}
To obtain the exact controllability, we need more assumptions on the low frequency part.
\begin{thm}\label{thm:thm of varying metric}
Given $T>0$, suppose that:
\begin{enumerate}
  \item $(\omega,T,p_{K_i})$ satisfies GCC, $i=1,2,\cdots,n$,
  \item \label{hyp2} $K_1>K_2>\cdots>K_n$ in $\omega$,
  \item $\Omega$ has no infinite order of tangential contact on the boundary,
  \item The system \cref{eq:control system of different metrics} satisfies the unique continuation property of eigenfunctions.
\end{enumerate}
Then the system \cref{eq:control system of different metrics} is exactly controllable in $(H^1_0(\Omega)\times L^2(\Omega))^n$.
\end{thm}

Now, we consider the particular case of constant coefficients. Define the diagonal matrix $D=\left(
\begin{array}{ccc}
     d_1& & \\
      &\ddots& \\
      & &d_n
\end{array}
\right)$ and $B=\left(
\begin{array}{c}
     b_1 \\
     \vdots\\
     b_n
\end{array}
\right)$.
We use $\Delta$ to denote the canonical Laplace operator. Now we consider the simultaneous control problem for the system:
\begin{equation}\label{eq:control system of constant coeff}
    \p_t^2 U-D\Delta U=Bf\mathbf{1}_{]0,T[}(t)\mathbf{1}_{\omega}(x)\text{ in }]0,T[\times\Omega,
\end{equation}
where $U=\left(
\begin{array}{cc}
     u_1 \\
     \vdots\\
     u_n
\end{array}
\right)$. This system can be written as 
\begin{equation*}
    \left\{
    \begin{array}{l}
        (\partial_t^2-d_1\Delta)u_1=b_1f\mathbf{1}_{]0,T[}(t)\mathbf{1}_{\omega}(x)\text{ in }]0,T[\times\Omega,\\
        \vdots\\
        (\partial_t^2-d_n\Delta)u_n=b_n f\mathbf{1}_{]0,T[}(t)\mathbf{1}_{\omega}(x)\text{ in }]0,T[\times\Omega,\\
        u_j=0\quad\text{ on }]0,T[\times\p\Omega,1\leq j\leq n,\\
    u_j(0,x)=u_j^0(x),\quad \p_tu_j(0,x)=u_j^1(x), 1\leq j\leq n.
    \end{array}
    \right.
\end{equation*}
First, we introduce the Kalman rank condition for the system \cref{eq:control system of constant coeff}.
\begin{defi}[Kalman rank condition]
Define $[D|B]=[D^{n-1}B|\cdots|DB|B]$. We say $(D,B)$ satisfies the Kalman rank condition if and only if $[D|B]$ is full rank.
\end{defi}
\begin{rem}
In our setting, $(D,B)$ satisfies the Kalman rank condition if and only if all $d_j$ are distinct and $b_j\neq0$, $1\leq j\leq n$(See \cite[Remark 1.1]{ammar2009kalman}).
\end{rem}
\begin{thm}\label{thm:constant result}
Given $T>0$, suppose that:
\begin{enumerate}
  \item $(\omega,T,p_{d_i})$ satisfies GCC, $i=1,\cdots,n$.
  \item $\Omega$ has no infinite order of tangential contact on the boundary.
\end{enumerate}
Then the system \cref{eq:control system of constant coeff} is exactly controllable in $(H^1_0(\Omega)\times L^2(\Omega))^n$ if and only if $(D,B)$ satisfies the Kalman rank condition.
\end{thm}
\begin{rem}
Let $T_0$ be the controllability time corresponding to the wave equation with unit speed of propagation. Then the controllability time in the \cref{thm:constant result} satisfies $T>T_0\max\{\frac{1}{\sqrt{d_j}};j=1,2,\cdots,n\}$.
\end{rem}

In advance, we consider the case with multiple control functions $f_1,f_2,\cdots,f_m(1\leq m\leq n)$. To be more specific, we consider the system:
\begin{equation}\label{eq:multi control system}
\left\{
\begin{array}{l}
     \p_t^2 U-D\Delta U=BF\mathbf{1}_{]0,T[}(t)\mathbf{1}_{\omega}(x)\text{ in }]0,T[\times\Omega,  \\
     U|_{\p\Omega}=0,\\
     (U,\p_t U)|_{t=0}=(U^0,U^1).
\end{array}
\right.
\end{equation}
where $D=diag(d_1,d_2,\cdots,d_n)$, $F=\left(
\begin{array}{cc}
     f_1 \\
     \vdots\\
     f_m
\end{array}
\right)$, and $B=\left(
            \begin{array}{ccc}
              b_{11} & \cdots & b_{1m} \\
              \vdots  & \ddots & \vdots \\
              b_{n1} & \cdots & b_{nm}  \\
           \end{array}
          \right)$. We can also define the Kalman rank condition $rank[D|B]=n$. Here we recall that $[D|B]=(D^{n-1}B|D^{n-2}B|\cdots|DB|B)$. We have the following theorem:
\begin{thm}\label{thm:mutli thm}
Given $T>0$, suppose that:
\begin{enumerate}
  \item $(\omega,T,p_{d_i})$ satisfies GCC, $i=1,\cdots,n$.
  \item $\Omega$ has no infinite order of contact on the boundary.
\end{enumerate}
Then the system \cref{eq:multi control system} is exactly controllable if and only if $(D,B)$ satisfies the Kalman rank condition.
\end{thm}

\section{Geometric Preliminaries}
\label{sec:Geometric Preliminaries}
Let $B=\{y\in\R^{d}:|y|<1\}$ be the unit ball in $\R^{d}$. In a tubular neighbourhood of the boundary, we can identify $M=\Omega\times\R_t$ locally as $[0,1[\times B$. More precisely, for $z\in \overline{M}=\overline{\Omega}\times\R_t$, we note that $z=(x,y)$, where $x\in [0,1[$ and $y\in B$ and $z\in\p M=\p\Omega\times\R_t$ if and only if $z=(0,y)$. Now we consider $R=R(x,y,D_y)$ which is a second order scalar, self-adjoint, classical, tangential and smooth pseudo-differential operator, defined in a neighbourhood of $[0,1]\times B$ with a real principal symbol $r(x,y,\eta)$, such that
\begin{equation}
    \frac{\partial r}{\partial \eta}\neq0\text{ for }(x,y)\in[0,1[\times B\text{ and }\eta\neq0.
\end{equation}
Let $Q_0(x,y,D_y)$, $Q_1(x,y,D_y)$ be smooth classical tangential pseudo-differential operators defined in a neighbourhood of $[0,1]\times B$, of order $0$ and $1$, and principal symbols $q_0(x,y,\eta)$, $q_1(x,y,\eta)$, respectively. Denote $P=(\partial^2_x+R)Id+Q_0\partial_x+Q_1$. The principal symbol of $P$ is 
\begin{equation}
    p=-\xi^2+r(x,y,\eta).
\end{equation}
We use the usual notations $TM$ and $T^*M$ to denote the tangent bundle and cotangent bundle corresponding to $M$, with the canonical projection $\pi$
\begin{equation*}
    \pi:TM(\text{ or }T^*M)\rightarrow M.
\end{equation*}
Denote $r_0(y,\eta)=r(0,y,\eta)$. Then we can decompose $T^*\partial M$ into the disjoint union $\mathcal{E}\cup\mathcal{G}\cup\mathcal{H}$, where
\begin{equation}
    \mathcal{E}=\{r_0<0\},\quad \mathcal{G}=\{r_0=0\},\quad \mathcal{H}=\{r_0>0\}. 
\end{equation}
The sets $\E$, $\G$, $\h$ are called elliptic, glancing, and hyperbolic set, respectively. 
Define $Char(P)=\{(x,y,\xi,\eta)\in T^*\R^{d+1}|_{\overline{M}}:\xi^2=r(x,y,\xi,\eta)\}$ to be the characteristic manifold of $P$. For more details, see \cite{measure-for-system} and \cite{MR1627111}.
\subsection{Generalised bicharacteristic flow}
We begin with the definition of the Hamiltonian vector field. For a symplectic manifold $S$ with local coordinates $(z,\zeta)$, a Hamiltonian vector field associated with a real valued smooth function $f$ is defined by the expression:
\begin{equation*}
    H_f=\frac{\p f}{\p\zeta}\frac{\p}{\p z}-\frac{\p f}{\p z}\frac{\p}{\p\zeta}.
\end{equation*}
Considering the principal symbol $p$, we can also consider the associated Hamiltonian vector field $H_p$. The integral curve of this Hamiltonian $H_p$, denoted by $\gamma$, is called a bicharacteristic of $p$. Our next goal is to study the behavior of the bicharacteristic near the boundary. To describe the different phenomena when a bicharacteristic approaches the boundary, we need a more accurate decomposition of the glancing set $\G$. Let $r_1=\p_x r|_{x=0}$. Then we can define the decomposition $\G=\bigcup_{j=2}^{\infty}\G^{j}$, with
\begin{equation*}
\begin{aligned}
    \G^2&=\{(y,\eta):r_0(y,\eta)=0,r_1(y,\eta)\neq0\},\\
    \G^3&=\{(y,\eta):r_0(y,\eta)=0,r_1(y,\eta)=0,H_{r_0}(r_1)\neq0\},\\
    &\vdots\\
    \G^{k+3}&=\{(y,\eta):r_0(y,\eta)=0,H^j_{r_0}(r_1)=0,\forall j\leq k,H^{k+1}_{r_0}(r_1)\neq0\},\\
    &\vdots\\
    \G^{\infty}&=\{(y,\eta):r_0(y,\eta)=0,H^j_{r_0}(r_1)=0,\forall j\}.
\end{aligned}
\end{equation*}
Here $H^j_{r_0}$ is just the vector field $H_{r_0}$ composed $j$ times. Moreover, for $\G^2$, we can define $\G^{2,\pm}=\{(y,\eta):r_0(y,\eta)=0,\pm r_1(y,\eta)>0\}$. Thus $\G^2=\G^{2,+}\cup\G^{2,-}$. For $\rho\in\G^{2,+}$, we say that $\rho$ is a gliding point and for $\rho\in\G^{2,-}$, we say that $\rho$ is a diffractive point. For $\rho\in\G^j$, $j\geq2$, we say that a bicharacterisric of $p$ tangentially contact the boundary $\{x=0\}\times B$ with order $j$ at the point $\rho$.

Consider a bicharacteristic $\gamma(s)$ with $\pi(\gamma(0))\in M$ and $\pi(\gamma(s_0))\in\p M$ be the first point which touches the boundary. Then if $\gamma(s_0)\in\h$, we can define $\xi^{\pm}(\gamma(s_0))=\pm\sqrt{r_0(\gamma(s_0))}$, which are the two different roots of $\xi^2=r_0$ at the point $\gamma(s_0)$. Notice that the bicharacteristic with the direction $\xi^-$ will leave the domain $M$ while the bicharacteristic with the other direction $\xi^+$ will enter into the interior of $M$. This leads to a definition of the broken bicharacteristics(See \cite{hormander1985analysis} Section 24.2 for more details):
\begin{defi}\label{defibica}
A broken bicharacteristic of $p$ is a map:
\begin{equation*}
    s\in I\backslash D\mapsto \gamma(s)\in T^*M\backslash\{0\}
\end{equation*}
where $I$ is an interval on $\R$ and $D$ is a discrete subset, such that
\begin{enumerate}
    \item If $J$ is an interval contained in $I\backslash D$, then for $s\in J\mapsto\gamma(s)$ is a bicharacteristic of $p$ in $M$.
    \item If $s\in D$, then the limits $\gamma(s^+)$ and $\gamma(s^-)$ exist and belongs to $T_z^*M\backslash\{0\}$ for some $z\in\p M$, and the projections in $T_z^*\p M\backslash\{0\}$ are the same hyperbolic point.
\end{enumerate}
\end{defi}
If $\gamma(s_0)\in\G$, we have different situations. If $\gamma(s_0)\in\G^{2,+}$, then $\gamma(s)$, locally near $s_0$, passes transversally and enters into $T^*M$ immediately. If $\gamma(s_0)\in\G^{2,-}$ or $\gamma(s_0)\in\G^k$ for some $k\geq3$, then $\gamma(s)$ will continue inside $T^*\p M$ and follow the Hamiltonian flow of $H_{-r_0}$. To be more precise, we have the definition of the generalized bicharacteristics(See \cite{hormander1985analysis} Section 24.3 for more details):
\begin{defi}
A generalized bicharacteristic of $p$ is a map:
\begin{equation*}
    s\in I\backslash D\mapsto\gamma(s)\in T^*M\cup\G
\end{equation*}
where $I$ is an interval on $\R$ and $D$ is a discrete subset $I$ such that $p\circ\gamma=0$ and the following properties hold:
\begin{enumerate}
    \item $\gamma(s)$ is differentiable and $\frac{\dif\gamma}{\dif s}=H_p(\gamma(s))$ if $\gamma(s)\in T^*M$ or $\gamma(s)\in\G^{2,+}$.
    \item Every $t\in D$ is isolated i.e. there exists $\epsilon>0$ such that $\gamma(s)\in T^*\overline{M}\backslash T^*\p M$ if $0<|s-t|<\epsilon$, and the limits $\gamma(s^{\pm})$ are different points in the same hyperbolic fiber of $T^*\p M$.
    \item $\gamma(s)$ is differentiable and $\frac{\dif\gamma}{\dif s}=H_{-r_0}(\gamma(s))$ if $\gamma(s)\in \G\backslash\G^{2,+}$.
\end{enumerate}
\end{defi}
\begin{rem}
We denote the Melrose cotangent compressed bundle by ${ }^bT^*\overline{M}$ and the associated canonical map by $j:$ $T^*\overline{M}\mapsto { }^bT^*\overline{M}$. $j$ is defined by
\begin{equation*}
    j(x,y,\xi,\eta)=(x,y,x\xi,\eta).
\end{equation*}
Under this map $j$, one could see $\gamma(s)$ as a continuous flow on the compressed cotangent bundle ${ }^bT^*\overline{M}$. This is the so-called Melrose-Sj\"ostrand flow. 
\end{rem}
From now on we always assume that there is no infinite tangential contact between the bicharacteristic of $p$ and the boundary. This is in the meaning of the following definition:
\begin{defi}
We say that there is no infinite contact between the bicharacteristics of $p$ and the boundary if there exists $N\in\N$ such that the gliding set $\G$ satisfies 
\begin{equation*}
    \G=\bigcup_{j=2}^{N}\G^{j}.
\end{equation*}
\end{defi}
It is well-known that under this hypothesis there exists a unique generalized bicharacteristic passing through any point. This means that the Melrose-Sj\"ostrand flow is globally well-defined. One can refer to \cite{melrose1978singularities} and \cite{melrose1982singularities} for the proof.

\section{High Frequency Estimates}\label{sec:High Frequency Estimates}
\subsection{Microlocal defect measure}\label{sec:Microlocal defect measure}
In this section, we introduce the microlocal defect measures based on the article by Gérard and Leichtnam~\cite{MR1233448} for Helmoltz equation and~Burq~\cite{MR1627111} for wave equations.  Let $(u^k)_{k\in\mathbb{N}}\in  L^2_{loc}(\mathbb{R}_t; L^2( \Omega))$ be a bounded sequence, converging weakly to $0$ and such that
\begin{equation}
\left\{
\begin{aligned}
    &Pu^k=o(1)_{H^{-1}},\\
    &u^k|_{\partial M}=0.
\end{aligned}
\right.
\end{equation}
Let $\underline{u}_k$ be the extension by $0$ across the boundary of $\Omega$. Then the sequence $\underline{u}_k$ is bounded in $L_{loc}^2( \mathbb{R}_t; L^2(\mathbb{R}^d))$.  
Let $\mathcal{A}$ be the space of classical polyhomogeneous pseudo-differential operators of order $0$ with compact support in $\R_t \times \mathbb{R}^d$ (i.e, $A=\varphi A\varphi$ for some $\varphi\in C^{\infty}_0(\R_t \times \mathbb{R}^d)$). Let us denote by $\mathcal{M}^+$ the set of non negative Radon measures on $T^* (\R_t \times \mathbb{R}^d)$. From~\cite[Section 1]{MR1627111}, we have  the existence of the microlocal defect measure as follows: 
\begin{prop}[Existence of the microlocal defect measure]\label{thm:existence of the measure}
There exists a subsequence of $(\underline{u}^k)$ (still noted by $(\underline{u}^k)$) and $\mu \in\M^+$ such that 
\begin{equation}
    \forall A\in\A,\quad 
    \lim_{k\rightarrow\infty}(A\underline{u}^k,\underline{u}^k)_{L^2}=\langle\mu,\sigma(A)\rangle,
\end{equation}
where $\sigma(A)$ is the principal symbol of the operator $A$ (which is a smooth function homogeneous of order $2$ in the variable $\xi$, i.e. a function on $S^* ((\R_t \times \mathbb{R}^d))$.
\end{prop}
\begin{rem}
In the article \cite{MR1385677}, Lebeau constructed the microlocal defect measure in another approach (see \cite[Appendice]{MR1385677} for more details). In the article \cite{measure-for-system}, Burq and Lebeau proved the similar existence result \cite[Proposition 2.5]{measure-for-system} in a setting of systems, which can be seen as an extension of \cref{thm:existence of the measure} 
\end{rem}
From \cite[Théorème 15]{MR1627111}, we have the following proposition.
\begin{prop}\label{thm:singular lemma}
For the microlocal defect measure $\mu$ defined above, we have the following properties. 
\begin{itemize}
\item 
The measure $\mu$ is supported on the intersection of the characteristic manifold with $\mathbb{R}_t \times \overline{\Omega}$,
\begin{equation}\label{support}
 \text{supp} ( \mu) \subset \{ (t,x,\tau, \xi);  x\in \overline{M}, \tau ^2 = {}^t\xi K(x) \xi \}.
 \end{equation}
\item The measure $\mu$ does not charge the hyperbolic points in $\partial M$,
$$ \mu( \mathcal{H}) =0. $$
\item The measure $\mu$ is invariant by the generalised bicharacteristic flow.
\end{itemize}
\end{prop}
\begin{rem} Notice first that in \cite[Section 3]{MR1627111}, the author considered the case of solutions to the wave equation at the energy level (bounded in $H^1_\text{loc}$, and hence was considering second order operators. However, it is easy to pass from $H^1$ to $L^2$ solutions by applying the operator $\partial_t$ and conversely from $L^2$ to $H^1$ by applying the operator $\partial_t^{-1}$, i.e. if $v$ is an $L^2$ solution, considering the solution $u$ associated to $\bigl( (-\Delta_D)^{-1} (\partial_t v\mid_{t=0}), v\mid_{t=0}\bigr)$, which of course satisfies $\partial_t u =v$. This procedure amounts to replacing the test operators of order $0$ $A$ by the test operator of order $2$, $B= -\partial_t\circ A\circ \partial_t$, but since $\tau^2$ does not vanish on the characteristic manifold, it is an elliptic factor which changes nothing.
\end{rem}
\begin{rem}
 Notice also that due to discontinuity of the generalised bicharacteristics when they reflect on the boundary at hyperbolic points (the points corresponding to the left and right limits at $s\in D$), in \cref{defibica}, the generalised bicharacteristic flow is not well defined (there are two points above any points corresponding to $s\in D$). However, since the measure $\mu$ does not charge these hyperbolic points, this flow is well defined $\mu$ almost surely and the invariance property makes sense. Notice also that in \cite[Appendice]{MR1627111},  weaker property than invariance (namely that the support is a union of generalised bicahracteristics) is proved. The general result follows from this weaker result by applying the strategy in ~\cite{MR1385677}. In any case, for the purpose of the present article, the invariance of the support would suffice.
\end{rem}

\subsection{Proof of the \cref{thm:except finte dim}}\label{sec:Proof of  thm:except finte dim}
Let $V=(v^0_1, v^1_1,\cdots, v^0_n,v_n^1)$. We introduce the following spaces:
\begin{itemize}
    \item We define $\K_1=(H^1_0(\Omega)\times L^2(\Omega))^n$ endowed with the norm
\begin{equation*}
    ||V||^2_{\mathcal{K}_1}=\sum_{j=1}^n\int_{\Omega}(K_j\nabla v^0_j\cdot\overline{\nabla v^0_j}+|v_i^1|^2)\kappa_i\dif x.
\end{equation*}
    \item We define $\K_0=(L^2(\Omega)\times H^{-1}(\Omega))^n$ endowed with the norm
$$
||V||^2_{\mathcal{K}_0}=\sum_{i=1}^n\int_{\Omega}|v_i^0|^2\kappa_idx+<v_i^1,T_{K_i}v_i^1>_{H^{-1},H_0^1},
$$
where 
\begin{equation*}
\begin{aligned}
   T_{K_i}:H^{-1}(\Omega)&\rightarrow H^1_0(\Omega)\\
   f&\mapsto w
\end{aligned}
\end{equation*}
is defined as the unique solution $w\in H^1_0(\Omega)$ to $-\frac{1}{\kappa_i}div(\kappa_iK_i\nabla T_{K_i}w)=f$.
    \item We define $\K_{-1}= (H^{-1}(\Omega)\times D(-\Delta)')^n$ endowed with the norm
$$
||V||^2_{\mathcal{K}_{-1}}=\sum_{i=1}^n<v_i^0,T_{K_i}v_i^0>_{H^{-1},H_0^1}+<v_i^1,\Tilde{T}_{K_i}v_i^1>_{D(-\Delta_{K_i})^*,D(-\Delta_{K_i})},
$$
where $D(-\Delta)$ is the domain of the Laplacian operator with zero Dirichlet boundary condition and $D(-\Delta)'$ is its dual space, and 
\begin{equation*}
\begin{aligned}
   \Tilde{T}_{K_i}:D(-\Delta)'&\rightarrow D(-\Delta)\\
   \Tilde{f}&\mapsto \Tilde{w}
\end{aligned}
\end{equation*}
is defined as the unique solution $\Tilde{w}\in D(-\Delta)$ to $(-\Delta_{K_i})^2\Tilde{T}_{K_i}\Tilde{w}=\Tilde{f}$.
\end{itemize}
\begin{rem}
For any $j\in\{1,2,\cdots,n\}$, $D(-\Delta_{K_j})=D(-\Delta)$.
\end{rem}
Recall the considered control system:
\begin{equation}\label{eq:control system of perp}
\left\{
  \begin{array}{l}
    \Box_{K_1}u_1=b_1f\mathbf{1}_{]0,T[}(t)\mathbf{1}_{\omega}(x)\text{ in }]0,T[\times\Omega,\\
    \Box_{K_2}u_2=b_2f\mathbf{1}_{]0,T[}(t)\mathbf{1}_{\omega}(x)\text{ in }]0,T[\times\Omega,\\
    \vdots\\
    \Box_{K_n}u_n=b_nf\mathbf{1}_{]0,T[}(t)\mathbf{1}_{\omega}(x)\text{ in }]0,T[\times\Omega,\\
    u_j=0\quad\text{ on }]0,T[\times\p\Omega,1\leq j\leq n,\\
    (u_1,\p_t u_1,\cdots,u_n,\p_t u_n)|_{t=0}=U(0).
  \end{array}
\right.
\end{equation}
Consider the homogeneous system:
\begin{equation}\label{eq:H-adjoint system perp}
\left\{
  \begin{array}{l}
    \Box_{K_1}v^h_1=0\text{ in }]0,T[\times\Omega,\\
    \Box_{K_2}v^h_2=0\text{ in }]0,T[\times\Omega,\\
    \vdots\\
    \Box_{K_n}v^h_n=0\text{ in }]0,T[\times\Omega,\\
    v^h_j=0\quad\text{ on }]0,T[\times\p\Omega,1\leq j\leq n,\\
    (v^h_1,\p_t v^h_1,\cdots,v^h_n,\p_t v^h_n)|_{t=0}=V^h(0)\in \K_1.
  \end{array}
\right.
\end{equation}
Now, let us define 
\begin{equation}\label{eq:vanishing space}
    E=\{V^h(0)\in\K_1:(b_1\kappa_1v^h_1+\cdots+b_n\kappa_nv^h_n)(t,x)=0\text{, for any }t\in]0,T[,x\in\omega\},
\end{equation}
where $(v^h_1,\cdots,v^h_n)$ is the solution to the homogeneous system \cref{eq:H-adjoint system perp}. Hence, $E$ is a closed subspace in $\K_1$. Denote the orthogonal projector operator $\mathbb{P}:\K_1\rightarrow E^{\perp}$.
And the adjoint system of System \cref{eq:control system of perp} is the following system:
\begin{equation}\label{eq:adjoint system perp}
\left\{
  \begin{array}{l}
    \Box_{K_1}v_1=0\text{ in }]0,T[\times\Omega,\\
    \Box_{K_2}v_2=0\text{ in }]0,T[\times\Omega,\\
    \vdots\\
    \Box_{K_n}v_n=0\text{ in }]0,T[\times\Omega,\\
    v_j=0\quad\text{ on }]0,T[\times\p\Omega,1\leq j\leq n,\\
    (v_1,\p_t v_1,\cdots,v_n,\p_t v_n)|_{t=0}=\mathbb{P}^*V(0)\in \K_0.
  \end{array}
\right.
\end{equation}
Using inequality \cref{eq:P-general ob}, the $\mathbb{P}-$exactly controllability of the system \cref{eq:control system of perp} is equivalent to proving the following observability inequality:
\begin{equation}\label{eq:ob-ineq}
C\int_0^T\int_{\omega}{|b_1\kappa_1v_1+\cdots+b_n\kappa_nv_n|}^2dxdt\geq ||\mathbb{P}^*V(0)||_{\K_0}^2,
\end{equation}
where $(v_1,\cdots,v_n)$ is the solution to the adjoint system \cref{eq:adjoint system perp}.

\subsubsection{Step 1: Establish a weak observability inequality}
First we want to prove a weak inequality:
\begin{equation}\label{eq:relaxed ob-inequality}
||\mathbb{P}^*V(0)||^2_{\K_0}\leq C\left(\int_0^T\int_{\omega}{|b_1\kappa_1v_1+\cdots+b_n\kappa_nv_n|}^2dxdt+||\mathbb{P}^*V(0)||^2_{\K_{-1}}\right),
\end{equation}
If the above inequality was false, we could get a sequence $(\mathbb{P}^*\widetilde{V}_0^k)_{k\in\N}$ such that
\begin{equation}\label{eq:normalized norm}
    ||\mathbb{P}^*\widetilde{V}_0^k||^2_{\K_0}=1,
\end{equation}

\begin{equation}\label{eq:norm-v1+v2}
\int_0^T\int_{\omega}|b_1\kappa_1v_1^k+\cdots+b_n\kappa_nv_n^k|^2dxdt\rightarrow0,k\rightarrow\infty,
\end{equation}
and \begin{equation}\label{eq:norm-V0}
||\mathbb{P}^*\widetilde{V}_0^k||^2_{\K_{-1}}\rightarrow0,k\rightarrow\infty.
\end{equation}
Here we use $v^k_i(1\leq i\leq n)$ to denote the corresponding solution of the system \cref{eq:adjoint system perp} with the initial data $\mathbb{P}^*\widetilde{V}_0^k$. Hence, we obtain $n$ bounded sequences $\{v^k_i\}_{k\in\N}(1\leq i\leq n)$. Let $\mu_i$ be the defect measure associated  to the sequence $\{v^k_i\}_{k\in\N}$,  by the construction in \Cref{sec:Microlocal defect measure}. Notice that in these constructions,  each sequence $\{v^k_i\}_{k\in\N}$ is solution to a particular wave equation 
$$ \Box_{K_i} v^k_i =0, v^k_i \mid _{\partial \Omega} =0
$$ 
and in \Cref{sec:Geometric Preliminaries} this corresponds to different principal symbols $p_i$, different sets $\mathcal{G}_i, \mathcal{H}_i, \mathcal{E}_i$ and different generalised bicharacteristic $\gamma_i$.

From the definition of the measures,  we obtain 
\begin{equation*}
    \forall A\in\A,\quad \langle\mu_i,\sigma(A)\rangle=\lim_{k\rightarrow\infty}(A\underline{v}_i^k,\underline{v}_i^k)_{L^2},
\end{equation*}
where $\underline{v}_i^k$ is the extension by $0$ across the boundary of $\Omega$. From \Cref{thm:singular lemma} we have 
\begin{lem}\label{thm:propagation lemma for measures}
Each measure $\mu_i$ is supported on the characteristic manifold
$$\text{Char($p_i$)} = \{ (t,x,\tau,\xi) \in T^*\mathbb{R} \times \mathbb{R}^d\mid_{\overline{\Omega}}; \tau^2 = {}^t\xi K_i(x) \xi \} $$
and is invariant along the generalised bicharacteristic flow associated to the symbol $p_i= ^t \xi K_i(x) \xi - \tau^2$
\end{lem}

\begin{lem}\label{thm:singular measure}
The measures $\mu_{i}$ and $\mu_{l}$ are mutually singular in $]0,T[\times\omega$, for $i\neq l$. 
\end{lem}
\begin{rem}
We recall that two measures $\mu$ and $\nu$ are singular if there exists a measurable set $A$ such that $\mu(A)=0$ and $\nu(A^c)=0$.
\end{rem}
\begin{proof}
This follows easily from ~\cref{thm:propagation lemma for measures} and the assumption~\ref{hyp2} in~\cref{thm:thm of varying metric}, which implies that over $\omega$, the two characteristic manifolds $ \text{Char}(p_i)$ and $\text{Char}(p_l)$ are disjoint.
\end{proof}

\begin{lem}\label{thm:mixed term lemma}
For $A\in\mathcal{A}$ with the compact support in $]0,T[\times\omega$, we obtain that for $i\neq l$:
\begin{equation}\label{eq:estimates for mixed terms}
    \limsup_{k\rightarrow\infty}|(A\underline{v}^k_i,\underline{v}^k_l)_{L^2}|=0.
\end{equation}
\end{lem}
\begin{proof}
For $\forall(t,x)\in]0,T[\times\omega$, we have that
\begin{equation*}
    \text{Char($p_i$)}\cap\text{Char($p_l$)}=\{0\},i\neq l.
\end{equation*}
Then we choose a cut-off function $\beta_i\in C^{\infty}(T^*\mathbb{R} \times \mathbb{R}^d)$ homogeneous of degree $0$ for $|(\tau, \xi)|\geq 1$,   with compact support in $]0,T[\times\omega$ such that
\begin{equation*}
    \beta_i|_{\text{Char($p_i$)}}=1, \beta_i|_{\text{Char($p_l$)}}=0, \text{ and }0\leq\beta_i\leq1.
\end{equation*}
Since $A\in\mathcal{A}$ with the compact support in $]0,T[\times\omega$, for some $\varphi\in C^{\infty}_0(]0,T[\times\omega)$, we have that $A=\varphi A\varphi$. We choose $\Tilde{\varphi}\in C^{\infty}_0(]0,T[\times\omega)$ such that $\Tilde{\varphi}|_{\text{supp} ( \varphi)}=1$ i.e, $\Tilde{\varphi}\varphi=\varphi$. Now let us consider the $(A\underline{v}^k_i,\underline{v}^k_l)_{L^2}$. First, we have that
\begin{equation*}
\begin{aligned}
    (A\underline{v}^k_i,\underline{v}^k_l)_{L^2} &=(\varphi A\varphi\underline{v}^k_i,\underline{v}^k_l)_{L^2} \\
    &=(\varphi A\varphi\underline{v}^k_i,\Tilde{\varphi}\underline{v}^k_l)_{L^2}\\
    &=((1-\Op{(\beta_i)})\varphi A\varphi\underline{v}^k_i,\Tilde{\varphi}\underline{v}^k_l)_{L^2}+(\Op{(\beta_i)}\varphi A\varphi\underline{v}^k_i,\Tilde{\varphi}\underline{v}^k_l)_{L^2}.
\end{aligned}
\end{equation*}
For the first term $((1-\Op(\beta_i))\varphi A\varphi\underline{v}^k_i,\Tilde{\varphi}\underline{v}^k_l)_{L^2}$, by the Cauchy-Schwarz inequality, therefore we obtain that
\begin{equation*}
  |((1-\Op(\beta_i))\varphi A\varphi\underline{v}^k_i,\Tilde{\varphi}\underline{v}^k_l)_{L^2}|\leq
  ||(1-\Op(\beta_i))\varphi A\varphi\underline{v}^k_i||_{L^2}||\Tilde{\varphi}\underline{v}^k_l||_{L^2}
\end{equation*}
As we know that $\{\underline{v}^k_l\}$ is bounded in $L^2_{loc}(\R_t\times\R^d)$, there exists a constant $C$ such that
\begin{equation*}
    ||\Tilde{\varphi}\underline{v}^k_l||^2_{L^2}=(\Tilde{\varphi}\underline{v}^k_l,\Tilde{\varphi}\underline{v}^k_l)_{L^2}\leq C.
\end{equation*}
From the definition of the measure $\mu_i$,  we obtain
\begin{equation*}
\begin{aligned}
    \lim_{k\rightarrow\infty}||(1-\Op(\beta_i))\varphi A\varphi\underline{v}^k_i||_{L^2}^2&=\lim_{k\rightarrow\infty}((1-\Op(\beta_i))\varphi A\varphi\underline{v}^k_i,(1-\Op(\beta_i))\varphi A\varphi\underline{v}^k_i)_{L^2}\\
    &=\poscalr{\mu_i}{(1-\beta_i)^2\varphi^4|\sigma(A)|^2}.
\end{aligned}
\end{equation*}
From \Cref{thm:singular lemma}, we have that $\text{supp }(\mu_i)\subset\text{Char($p_i$)}$. In addition, by the choice of $\beta_i$, we know that $1-\beta_i\equiv0$ on $\text{supp }(\mu_i)$, which implies that $\poscalr{\mu_i}{(1-\beta_i)^2\varphi^4|\sigma(A)|^2}=0$. Hence, we obtain 
\begin{equation}\label{eq:l-bichar-part}
   \limsup_{k\rightarrow\infty}|((1-\Op(\beta_i))\varphi A\varphi\underline{v}^k_i,\Tilde{\varphi}\underline{v}^k_l)_{L^2}|=0.
\end{equation}
The other term $(\Op(\beta_i)\varphi A\varphi\underline{v}^k_i,\Tilde{\varphi}\underline{v}^k_l)_{L^2}=(\underline{v}^k_i, \varphi A^*\varphi\Op(\beta_i)^*\Tilde{\varphi}\underline{v}^k_l)_{L^2} $ is dealt with similarly by exchanging $i$ and $l$.
\end{proof}

Now let us come back to the proof of the weak observability inequality \cref{eq:relaxed ob-inequality}. By the assumption \cref{eq:norm-v1+v2}, We know that
\begin{equation*}
    \int_0^T\int_{\omega}|b_1\kappa_1v^k_1+\cdots+b_n\kappa_nv_n^k|^2dxdt\rightarrow0,
\end{equation*}
for $\chi\in C^{\infty}_0(\omega\times]0,T[)$, and we would like to obtain:
\begin{equation*}
    \sum_{1\leq i,l\leq n}\poscalr{\chi b_i\kappa_iv^k_i}{\chi b_l\kappa_lv^k_l}\rightarrow0,\text{ as }k\rightarrow\infty.
\end{equation*}
According to \Cref{thm:mixed term lemma}, we know that for $i\neq l$,
\begin{equation}\label{eq:explicit mixed terms}
    \limsup_{k\rightarrow\infty}|\poscalr{\chi b_i\kappa_iv^k_i}{\chi b_l\kappa_lv^k_l}|=0.
\end{equation}
As a consequence, we know that
\begin{equation}
    \limsup_{k\rightarrow\infty}\Sigma_{i=1}^n\poscalr{\chi b_i\kappa_iv^k_i}{\chi b_i\kappa_iv^k_i}=0.
\end{equation}
Using again the definition of the measure $\mu_i$, we obtain the following:
\begin{equation}
    0\leq \poscalr{\mu_i}{(\chi b_i \kappa_i)^2}=\lim_{k\rightarrow\infty}\poscalr{\chi b_i\kappa_iv^k_i}{\chi b_i\kappa_iv^k_i}\leq\limsup_{k\rightarrow\infty}\Sigma_{i=1}^n\poscalr{\chi b_i\kappa_iv^k_i}{\chi b_i\kappa_iv^k_i}=0.
\end{equation}

Thus, we know that
\begin{equation*}
    \mu_{i}|_{\omega\times]0,T[}=0.
\end{equation*}
Since $\mu_{i}$ is invariant along the general bicharacteristics of $p_{K_i}$ (by \Cref{thm:propagation lemma for measures}), combining with GCC, we know that $\mu_i\equiv0$. Since $\mu_{i}=0$, we have $v_i^k\rightarrow0$ strongly in $L^2_{loc}(]0,T[\times\Omega)$. Now we have to estimate $||\partial_t v_1^k(0)||_{H^{-1}}$. Let $\chi\in C^{\infty}_0(]0,T[)$. Multiply the equation 
\begin{equation*}
    \Box_{K_1}v_1=0
\end{equation*}
by $T_{K_1}(\chi^2 v_1^k)$ and then integrate on $]0,T[\times\Omega$. We obtain that
\begin{equation}
\begin{aligned}
    0&=\int_0^T\int_\Omega\Box_{K_1}v_1^k\cdot\overline{T_{K_1}(\chi^2 v_1^k)}\dif x\dif t\\
    &=\int_0^T\int_\Omega v_1^k\cdot\overline{(-\Delta_{K_1})T_{K_1}(\chi^2 v_1^k)}\dif x\dif t-\int_0^T\int_\Omega \p_tv_1\cdot\overline{T_{K_1}(\p_t(\chi^2) v_1^k)}\dif x\dif t\\
    &-\int_0^T||\chi\p_t v^k_1||^2_{H^{-1}}\\
    &=||\chi v^k_1||^2_{L^2}-\int_0^T||\chi\p_t v^k_1||^2_{H^{-1}}+\int_0^T\int_\Omega v_1^k\cdot\overline{T_{K_1}(\p^2_t(\chi^2) v_1^k+\p_t(\chi^2)\p_t v_1^k)}\dif x\dif t
\end{aligned}
\end{equation}
For the term $\int_0^T\int_\Omega v_1^k\cdot\overline{T_{K_1}(\p^2_t(\chi^2) v_1^k+\p_t(\chi^2)\p_t v_1^k)}\dif x\dif t$, we know that $v_1^k\rightarrow0$ strongly in $L^2_{loc}(]0,T[\times\Omega)$ and $T_{K_1}(\p^2_t(\chi^2) v_1^k+\p_t(\chi^2)\p_t v_1^k)$ is bounded in $L^2$. Thus, up to a subsequence, it tends to $0$ as $k\rightarrow\infty$. Hence, we obtain that:
\begin{equation*}
   \int_0^T||\chi\partial_t v_1^k||^2_{H^{-1}}\rightarrow0\text{, as }k\rightarrow\infty.
\end{equation*}
So for all $0<t_1<t_2<T$, 
$$\int_{t_1}^{t_2}||\partial_t v_1^k(t)||^2_{H^{-1}}dt\rightarrow0.$$
So for almost every $t\in]t_1,t_2[$, $||\partial_t v_1^k(t)||^2_{H^{-1}}+||v_1^k(t)||^2_{L^2}\rightarrow0$. Then by the backward well-posedness, we can conclude:
\begin{equation*}
    ||\partial_t v_1^k(0)||^2_{H^{-1}}+||v_1^k(0)||^2_{L^2}\rightarrow0.
\end{equation*}
The same reasoning holds for $v_j^k$, $2\leq j\leq n$. This gives a contradiction with \cref{eq:normalized norm}, which proves the weak observability inequality \cref{eq:relaxed ob-inequality}.

\subsubsection{Step 2: Descriptions of the space $E$}
Define 
\begin{equation*}
    \mathcal{N}(T)=\{\mathbb{P}^*V(0)\in\K_0:(b_1\kappa_1v_1+\cdots+b_n\kappa_nv_n)(t,x)=0,\text{ for }t\in]0,T[,x\in\omega\}.
\end{equation*}
\begin{lem}
$E=\mathcal{N}(T)$ where $E$ was defined in \cref{eq:vanishing space} and $E$ has a finite dimension.
\end{lem}
\begin{proof}
According to the weak observability inequality \cref{eq:relaxed ob-inequality}, for $\mathbb{P}^*V(0)\in\mathcal{N}(T)$, we obtain that
\begin{equation}\label{eq:H-inequality}
||\mathbb{P}^*V(0)||^2_{\K_0}\leq C||\mathbb{P}^*V(0)||^2_{\K_{-1}}.
\end{equation}
We know that $\mathcal{N}(T)$ is a closed subspace of $\K_0$. By the compact embedding $\K_0\hookrightarrow \K_{-1}$, we know that $\mathcal{N}(T)$ has a finite dimension. By definition, we know that $E\subset\mathcal{N}(T)$. Hence, we obtain that $E$ has a finite dimension. Then we want to show that $E=\mathcal{N}(T)$. Define
\begin{equation*}
  \mathscr{A}=\left(
            \begin{array}{cccc}
              0 & -1 & \cdots& 0 \\
              -\Delta_{K_1}  & 0 & \cdots & 0 \\
              \vdots & \vdots & 0 & -1 \\
              0 & 0 &  -\Delta_{K_n} & 0 \\
            \end{array}
          \right)  .
\end{equation*}
Thus, the solution $(v_1,\p_t v_1,\cdots,v_n,\p_t v_n)^t$ can be written as 
\begin{equation*}
\left(
\begin{array}{c}
     v_1 \\
     \p_t v_1\\
     \vdots\\
     v_n\\
     \p_t v_n
\end{array}
\right)=e^{-t\mathscr{A}}\mathbb{P}^*V(0).
\end{equation*}
Since $\mathcal{N}(T)$ is of finite dimension, it is complete for any norm. Setting $\delta>0$, we know that \cref{eq:H-inequality} is still true for $\mathbb{P}^*V(0)\in\mathscr{N}(T-\delta)$. Taking $\mathbb{P}^*V(0)\in\mathscr{N}(T)$, for $\epsilon\in]0,\delta[$, we have $e^{-\epsilon\mathscr{A}}\mathbb{P}^*V(0)\in\mathscr{N}(T-\delta)$. For $\alpha$ large enough, as $\epsilon\rightarrow0^+$, 
\begin{equation*}
    (\alpha+\mathscr{A})^{-1}\frac{1}{\epsilon}(Id-e^{-\epsilon\mathscr{A}})\mathbb{P}^*V(0)\rightarrow\mathscr{A}(\alpha+\mathscr{A})^{-1}\mathbb{P}^*V(0).
\end{equation*}
As a consequence, we obtain $\mathscr{N}(T)\subset D(\mathscr{A})\subset\K_1$. Hence, we obtain that $E=\mathscr{N}(T)$ and has a finite dimension.
\end{proof}
\subsubsection{Step 3: Proof of the observability inequality \cref{eq:ob-ineq}}
If \cref{eq:ob-ineq} was false, we could find a sequence $\{\mathbb{P}^*V^k(0)\}_{k\in\N}\subset\K_0$ such that
\begin{equation}\label{eq:sequence}
||\mathbb{P}^*V^k(0)||_{\K_0}=1,\quad\int_0^T||b_1\kappa_1v^k_1+\cdots+b_n\kappa_nv_n^k||^2_{L^2(\omega)}dt\rightarrow0.
\end{equation}
First, we know that $\{\mathbb{P}^*V^k(0)^k\}_{k\in\N}$ is bounded in $\K_0=(L^2\times H^{-1})^n$. Hence, there exists a subsequence (also denoted by $\mathbb{P}^*V^k(0)$) weakly converging in $\K_0=(L^2\times H^{-1})^n$, to a limit which we denote with $\mathbb{P}^*V(0)$. We also know that $\mathbb{P}^*V(0)$ leads to a solution $(v_1,\cdots,v_n)$ of the system \cref{eq:adjoint system perp} and satisfies that $b_1\kappa_1v_1+\cdots+b_n\kappa_n v_n=0$ in $]0,T[\times\omega$. Thus, we know that $\mathbb{P}^*V(0)\in\mathscr{N}(T)=E$, which implies that $\mathbb{P}^*V(0)=0$. Since the embedding $\K_0\hookrightarrow \K_{-1}$ is compact, we obtain that $||\mathbb{P}^*V(0)^k||^2_{\K_{-1}}\rightarrow ||\mathbb{P}^*V(0)||^2_{\K_{-1}}$. From the weak observability inequality \cref{eq:relaxed ob-inequality}, we obtain:
\begin{equation*}
    1\leq C||\mathbb{P}^*V(0)||^2_{\K_{-1}},
\end{equation*}
which contradicts to the fact that $\mathbb{P}^*V(0)=0$. Then observability inequality \cref{eq:ob-ineq} follows. This concludes the proof of the $\mathbb{P}-$exact controllability of the system \cref{eq:control system of perp}.
\subsection{The Proof of \cref{thm:thm of varying metric}}
\label{sec:The proof of thm:thm of varying metric}
According to the proof above, we only need to show that $E^{\perp}=\{0\}$, which is equivalent to $\mathbb{P}^*=Id$. If we denote by $\Tilde{V}(t)$ the solution of 
\begin{equation*}
    \partial_t\Tilde{V}+\mathscr{A}\Tilde{V}=0,\Tilde{V}|_{t=0}=V(0),
\end{equation*}
then, $\mathscr{A}V(0)=-\partial_t\Tilde{V}|_{t=0}\in\mathscr{N}(T)$ provided that $V(0)\in\mathscr{N}(T)$. This implies that $\mathscr{A}\mathscr{N}(T)\subset\mathscr{N}(T)$. Since $\mathscr{N}(T)$ is a finite dimensional closed subspace of $D(\mathscr{A})$, and stable by the action of the operator $\mathscr{A}$, it contains an eigenfunction of $\mathscr{A}$. To be specific, there exists $(e_1,e_2,\cdots,e_n)\in\mathscr{N}(T)$ and $\lambda\in\C$ such that
\begin{equation*}
            \left(
            \begin{array}{cccc}
              0 & -1 & \cdots& 0 \\
              -\Delta_{K_1}  & 0 & \cdots & 0 \\
              \vdots & \vdots & 0 & -1 \\
              0 & 0 &  -\Delta_{K_n} & 0 \\
            \end{array}
          \right)
          \left(
                          \begin{array}{c}
                            e^0_1 \\
                            e^1_1 \\
                           \vdots\\
                            e^0_n \\
                            e^1_n\\
                          \end{array}
                        \right)
                        =\lambda\left(
                          \begin{array}{c}
                            e^0_1 \\
                            e^1_1 \\
                           \vdots\\
                            e^0_n \\
                            e^1_n\\
                          \end{array}
                        \right).
\end{equation*}
It is equivalent to the following system:
\begin{equation}
\left\{
  \begin{array}{l}
    -e^1_1=\lambda e^0_1\text{ in }\Omega,\\
    -\Delta_{K_1}e^0_1=\lambda e^1_1\text{ in }\Omega, \\
\cdots\\
-e^1_n=\lambda e^0_n\text{ in }\Omega,\\
 -\Delta_{K_n}e^0_n=\lambda e^1_n\text{ in }\Omega,\\
    b_1\kappa_1e^0_1+\cdots+b_n\kappa_n e^0_n=0,\text{ in } \omega. \\
  \end{array}
\right.
\end{equation}
We can simplify this into 
\begin{equation*}
\left\{
  \begin{array}{l}
    \Delta_{K_1}e^0_1=\lambda^2e^0_1\text{ in }\Omega,\\
    \Delta_{K_2}e^0_2=\lambda^2e^0_2\text{ in }\Omega,\\
\cdots\\
 \Delta_{K_n}e^0_n=\lambda^2e^0_n\text{ in }\Omega,\\
    b_1\kappa_1e^0_1+\cdots+b_n\kappa_ne^0_n=0\text{ in } \omega,\\
  \end{array}
\right.
\end{equation*}

Since the system satisfies the unique continuation of eigenfunctions, we know that $e^0_1=\cdots=e^0_n=0$ in $\Omega$, which implies that $E=\mathcal{N}(T)=\{0\}$. Hence, from \cref{eq:ob-ineq} with $\mathbb{P}^*=Id$, we obtain the observability inequality
\begin{equation*}
    C\int_0^T\int_{\omega}{|b_1\kappa_1v_1+\cdots+b_n\kappa_nv_n|}^2dxdt\geq ||V(0)||_{\K_0}^2.
\end{equation*}
This concludes the proof of \cref{thm:thm of varying metric}.

\section{Unique continuation of eigenfunctions}\label{sec:Unique continuation of eigenfunctions}
\subsection{A counterexample}\label{sec:A counterexample}
First, we construct an example to show that the conditions in \cref{thm:except finte dim} are not sufficient to ensure the unique continuation of eigenfunctions. Now, let us focus on the unique continuation problem in dimension $1$. We consider a smooth metric in dimension 1, $g=c(x)dx^2$. Then we can define the Laplace-Beltrami operator in the sense:
\begin{equation}
\begin{aligned}
     \Delta_g&=\frac{1}{\sqrt{\det(g)}}\frac{d}{dx}(\sqrt{\det(g)}g^{-1}\frac{d}{dx})\\
     &=\frac{1}{c}\frac{d^2}{dx^2}-\frac{c'}{2c^2}\frac{d}{dx}
\end{aligned}
\end{equation}
Fix the open interval $]0,\pi[$ and the subinterval $]a,b[\subset ]0,\pi[$($a>\frac{\pi}{2}$). Now we consider the unique continuation problem:
\begin{equation}\label{eq:UC}
\left\{
\begin{array}{cc}
     u''_1=-\lambda^2u_1,  \\
     \Delta_gu_2=-\lambda^2u_2,\\
     u_1+u_2=0 \text{ in }]a,b[,\\
     u_1,u_2\in H^1_0(]0,\pi[).
\end{array}
\right.
\end{equation}
In general, the unique continuation of eigenfunctions does not hold.
\begin{thm}\label{thm:counterexample}
There exists a smooth Riemannian metric $g=c(x)dx^2$, and two eigenfunctions $u_1$, $u_2$ of $\Delta_g$ and $\frac{d^2}{dx^2}$ on $]0,\pi[$ associated with eigenvalue $1$ such that $u_1+u_2=0$, in $]a,b[\subset ]0,\pi[$ and $u_1+u_2\not\equiv0$ in $]0,\pi[$.
\end{thm}
\begin{proof}Let $\chi\in C^{\infty}(\R)$ satisfying the following conditions:
\begin{enumerate}
    \item $\chi(0)=\chi(\pi)=0$;
    \item $0<\chi\leq K$ on $]0,\pi[$ and $\chi(\frac{\pi}{2})=K>1$;
    \item $\chi(x)=1,\forall x\in ]a,b[$;
    \item $\chi'(x)>0$ for $x\in[0,\frac{\pi}{2}[$, $\chi'(x)<0$ for $x\in]b,\pi]$ and $\chi'(x)<0$ for $x\in]\frac{\pi}{2},a[$
\end{enumerate}
Define $u_2(x)=-\chi(x)\sin{x}$. Hence, we obtain $u_2(x)=-\sin{x}$ on $]a,b[$ and \\$u_2'(x)=-\chi'(x)\sin{x}-\chi(x)\cos{x}$. Then we define $c(x)$ by
\begin{equation}
    c(x)=\frac{(\chi'(x)\sin{x}+\chi(x)\cos{x})^2}{K^2-\chi^2\sin^2{x}},
\end{equation}
with a constant $K>1$.
It is easy to check that $c\geq0$. Since we want $g$ to be a Riemannian metric, we need $c>0$. Let us discuss in different cases, 
\begin{enumerate}
    \item if $x\in ]0,\frac{\pi}{2}[$, we know that $\chi'(x)>0$, $\chi(x)>0$. Hence, we have $\chi'(x)\sin{x}+\chi(x)\cos{x}>0$;
    \item if $x\in[a,b]$, $\chi'(x)=0$, $\chi(x)=1$, we obtain $\chi'(x)\sin{x}+\chi(x)\cos{x}=\cos{x}<0$ since $a>\frac{\pi}{2}$;
    \item if $x\in]b,\pi[$, we know that $\chi'(x)<0$, $\chi(x)>0$. Hence, we have $\chi'(x)\sin{x}+\chi(x)\cos{x}<0$;
    \item if $x\in ]\frac{\pi}{2},a[$, we know that $\chi'(x)<0$, $\chi(x)>0$. Hence, we have $\chi'(x)\sin{x}+\chi(x)\cos{x}<0$;
    \item if $x=\frac{\pi}{2}$, $\chi'(\frac{\pi}{2})=0$, $c(\frac{\pi}{2})=1-\frac{\chi''(\frac{\pi}{2})}{K}\geq1$.
\end{enumerate}
So we can conclude that $c>0$ and $g$ is a Riemannian metric.

We want to show that $c$ is $C^{\infty}$ near $\frac{\pi}{2}$. Let $f(x)=(\chi'(x)\sin{x}+\chi(x)\cos{x})^2$ and $g(x)=K^2-\chi^2\sin^2{x}$, then we obtain $c(x)=\frac{f}{g}$. We claim that there exist $\tilde{f},\tilde{g}\in C^{\infty}$ and $\tilde{f}(\frac{\pi}{2})\neq0$, $\tilde{g}(\frac{\pi}{2})\neq0$ such that $f(x)=(x-\frac{\pi}{2})^2\tilde{f}(x)$ and $g(x)=(x-\frac{\pi}{2})^2\tilde{g}(x)$. We just use the Taylor expansion of $\chi$, $\chi'$, $\sin$ and $\cos$:
\begin{equation}
\begin{aligned}
\chi(x)&=K+\frac{1}{2}\chi''(\frac{\pi}{2})(x-\frac{\pi}{2})^2+R_1(x),\\
\chi'(x)&=\chi''(\frac{\pi}{2})(x-\frac{\pi}{2})+\frac{1}{2}\chi'''(\frac{\pi}{2})(x-\frac{\pi}{2})^2+R_2(x),\\
\sin(x)&=1-\frac{1}{2}(x-\frac{\pi}{2})^2+R_3(x),\\
\cos(x)&=-(x-\frac{\pi}{2})+R_4(x),\\
\end{aligned}
\end{equation}
where $\lim_{x\rightarrow\frac{\pi}{2}}\frac{R_j}{(x-\frac{\pi}{2})^2}=0$, for $j=1,2,3,4$.
Then we obtain:
\begin{equation}
\begin{aligned}
f(x)&=((\chi''(\frac{\pi}{2})-K)^2+\tilde{R_1})(x-\frac{\pi}{2})^2;\\
g(x)&=(-K(\chi''(\frac{\pi}{2})-K)+\tilde{R_2})(x-\frac{\pi}{2})^2.\\
\end{aligned}
\end{equation}
Here $\lim_{x\rightarrow\frac{\pi}{2}}\tilde{R_j}=0$ for $j=1,2$. Now if we choose a small neighbourhood of $\frac{\pi}{2}$, then $\tilde{f}=(\chi''(\frac{\pi}{2})-K)^2+\tilde{R_1}$ and $\tilde{g}=-K(\chi''(\frac{\pi}{2})-K)+\tilde{R_2}$ satisfy the property. So we know $c$ is $C^{\infty}$ and $c>0$, which means that $g$ is a smooth Riemannian metric. In addition, $c<1$ in $]a,b[$ and $\Delta_g$ and $\Delta$ admit the same eigenfunction in this interval $]a,b[$.
\end{proof}
\begin{rem}
In fact, we can construct a counterexample in any dimension $d\geq1$. For example, we define $M=]0,\pi[\times \Pi_y^{d-1}$ where $\Pi_y^{d-1}$ is the torus of dimension $d-1$. Then consider two metric $g_1=\dif x^2+\sum_{j=0}^{d-1}\dif y^2_j$ and $g_2=c(x)\dif x^2+\sum_{j=0}^{d-1}\dif y^2_j$ where $c(x)\dif x^2$ is the metric we constructed in the dimension 1. Take the same $u_1(x)$ and $u_2(x)$ in the proof of \cref{thm:counterexample}. Let $V$ be the eigenfunction of $\sum_{j=1}^{d-1}\frac{\dif^2}{\dif y_j^2}$ associated with eigenvalue $\alpha$ in $\Pi_y^{d-1}$. Then
\begin{equation*}
\left\{
\begin{array}{l}
      -\Delta_{g_1}(u_1(x)V(y))=(\alpha+1)u_1(x)V(y),\\
    -\Delta_{g_2}(u_2(x)V(y))=(\alpha+1)u_2(x)V(y),\\
    u_1(x)V(y)+u_2(x)V(y)=0,\text{ in }]a,b[\times\Pi_y^{d-1},\\
    u_1(x)V(y),u_2(x)V(y)\in H^1_0(M).
\end{array}
\right.
\end{equation*}
But we know $u_1(x)V(y)+u_2(x)V(y)\not\equiv0$ in $M$.
\end{rem}

As we have seen, not every smooth metric can give us the unique continuation of eigenfunctions. Here, we will give a positive result under a strong condition of analyticity. In particular, let us consider the example of two equations:
\begin{equation}\label{eq:2-eq}
    \left\{
  \begin{array}{l}
    \Box_{K_1}u_1=b_1f\mathbf{1}_{]0,T[}(t)\mathbf{1}_{\omega}(x)\text{ in }]0,T[\times\Omega\\
    \Box_{K_2}u_2=b_2f\mathbf{1}_{]0,T[}(t)\mathbf{1}_{\omega}(x)\text{ in }]0,T[\times\Omega\\
    u_j=0\quad\text{ on }]0,T[\times\p\Omega,j=1,2,\\
    u_j(0,x)=u_j^0(x),\quad \p_tu_j(0,x)=u_j^1(x), j=1,2.
  \end{array}
\right.
\end{equation}
\begin{prop}\label{thm:prop 2-eq}
Given $T>0$, suppose that:
\begin{enumerate}
  \item $(\omega,T,p_{K_i})$ satisfies GCC, $i=1,2$.
  \item $K_1>K_2$ in $\Omega$ with analytic coefficients.
  \item There exists a constant $c$ such that density functions $\kappa_1$, $\kappa_2$ are analytic and $\kappa_1=c\kappa_2$.
  \item $\Omega$ has no infinite order of contact on the boundary.
\end{enumerate}
Then the system \cref{eq:2-eq} is exactly controllable.
\end{prop}
\begin{proof}
According to \cref{thm:except finte dim}, we only need to show the unique continuation of eigenfunctions of system \cref{eq:2-eq}:
\begin{equation}
\left\{
  \begin{array}{l}
    -\Delta_{K_1}u_1=\lambda^2u_1\text{ in }\Omega, \\
    -\Delta_{K_2}u_2=\lambda^2u_2 \text{ in }\Omega,\\
    cu_1+u_2=0\text{ in } \omega.\\
  \end{array}
\right.
\end{equation}
Since $K_1$ and $K_2$ have analytic coefficients, we know $u_1$ and $u_2$ are analytic functions. Then $cu_1+u_2$ is also analytic. By unique continuation for analytic functions, $cu_1+u_2=0$ in the whole domain $\Omega$. By the relations of two density functions $\kappa_1=c\kappa_2$, we have:
\begin{equation}
\begin{aligned}
    \Delta_{K_1}u_1&=\frac{1}{\kappa_1(x)}div(\kappa_1(x)K_1\nabla u_1)\\
                   &=\frac{1}{c\kappa_2(x)}div(c\kappa_2(x)K_1\nabla u_1)\\
                   &=\frac{1}{\kappa_2(x)}div(\kappa_2(x)K_1\nabla u_1).
\end{aligned}
\end{equation}
Then 
\begin{equation*}
\begin{aligned}
    -c\Delta_{K_1}u_1-\Delta_{K_2}u_2&=-\frac{c}{\kappa_2(x)}div(\kappa_2(x)K_1\nabla u_1)-\frac{1}{\kappa_2(x)}div(\kappa_2(x)K_2\nabla u_2)\\
                        &=-\frac{c}{\kappa_2(x)}div(\kappa_2(x)K_1\nabla u_1)+\frac{c}{\kappa_2(x)}div(\kappa_2(x)K_2\nabla u_1)\\
                        &=-\frac{c}{\kappa_2(x)}div(\kappa_2(x)(K_1-K_2)\nabla u_2).
\end{aligned}
\end{equation*}
On the other hand, we know $-c\Delta_{K_1}u_1-\Delta_{K_2}u_2=\lambda^2(cu_1+u_2)=0$. Hence, we have:
\begin{equation*}
    -\frac{1}{\kappa_2(x)}div(\kappa_2(x)(K_1-K_2)\nabla u_1)=0.
\end{equation*}
We recall that $-\frac{1}{\kappa_2(x)}div(\kappa_2(x)(K_1-K_2)\nabla\cdot)$ is an elliptic operator. Hence, with $u_1|_{\p\Omega}=0$ on the boundary, we know that $u_1=0$. Hence, we deduce $u_2=-cu_1=0$ in $\Omega$, which gives $\mathcal{N}(T)=0$. 
\end{proof}

\subsection{Constant Coefficient Case}\label{sec:Constant Coefficient Case}
In this section, we consider the simultaneous control problem for the system:
\begin{equation}
    \p_t^2 U-D\Delta U=Bf\mathbf{1}_{]0,T[}(t)\mathbf{1}_{\omega}(x)\text{ in }]0,T[\times\Omega,
\end{equation}
where $U=\left(
\begin{array}{cc}
     u_1 \\
     \vdots\\
     u_n
\end{array}
\right)$, $B=\left(
\begin{array}{cc}
     b_1 \\
     \vdots\\
     b_n
\end{array}
\right)$ and $D=diag(d_1,\cdots,d_n)$. 
Then the system can be written as 
\begin{equation*}
    \left\{
    \begin{array}{l}
        (\partial_t^2-d_1\Delta)u_1=b_1f\mathbf{1}_{]0,T[}(t)\mathbf{1}_{\omega}(x)\text{ in }]0,T[\times\Omega,\\
        \vdots\\
        (\partial_t^2-d_n\Delta)u_n=b_n f\mathbf{1}_{]0,T[}(t)\mathbf{1}_{\omega}(x)\text{ in }]0,T[\times\Omega,\\
        u_j=0\quad\text{ on }]0,T[\times\p\Omega,1\leq j\leq n,\\
    u_j(0,x)=u_j^0(x),\quad \p_tu_j(0,x)=u_j^1(x), 1\leq j\leq n.
    \end{array}
    \right.
\end{equation*}
Recall that the Kalman rank condition for this case is $rank[D|B]=n$ if and only if all $d_j$ are distinct and $b_j\neq0$, $1\leq j\leq n$(See \cite{ammar2009kalman}). Without loss of generality, we may assume that $d_1<d_2<\cdots<d_n$. We want to prove the exact controllability for this case(\cref{thm:constant result}). 
\subsection{Proof of \cref{thm:constant result}}\label{sec:Proof of thm:constant result}
By \cref{thm:except finte dim}, we only need to prove the unique continuation properties for eigenfunctions. Here we only state some facts without repeating the same trick as before. Define 
\begin{equation*}
    \mathscr{N}(T)=\{\mathscr{V}\in (L^2\times H^{-1})^n:(b_1v_1+b_2v_2+\cdots+b_nv_n)(x,t)=0,\forall(x,t)\in]0,T[\times\omega\}.
\end{equation*}
Then, $\mathscr{N}(T)$ is a finite dimensional closed subspace of $D(\mathscr{A})$, and stable by the action of the operator $\mathscr{A}$, it contains an eigenfunction of $\mathscr{A}$, where $\mathscr{A}=\left(
\begin{array}{cccc}
     0&-Id\\
     -D\Delta& 0\\
\end{array}
\right)$. Thus there exist $\beta\in \C$ and $\mathscr{V}_{\beta}=(V_1,V_2)$ such that $\mathscr{A}\mathscr{V}_{\beta}=\beta\mathscr{V}_{\beta}$, i.e.
\begin{equation}
    -\Delta V_1= -\beta^2D^{-1}V_1
\end{equation}
If $\beta\neq0$, $(-\beta^2)^{-k}(-\Delta)^kV_1= D^{-k}V_1$ and $(-\Delta)^kB^tV_1= (-\beta^2)^kB^tD^{-k}V_1$. Since $V_1$ solves the Laplace eigenvalue problem, we know that $V_1$ is analytic in $\Omega$ which ensures that $B^tV_1=b_1v_1^1+\cdots+b_nv_1^n=0$ in the whole domain $\Omega$. Thus
\begin{equation}
0=[B^tV_1|(-\beta^2)^{-1}(-\Delta)B^tV_1|\cdots|(-\beta^2)^{-n}(-\Delta)^nB^tV_1]=[D|B]^tD^{1-n}V_1
\end{equation}
Since $rank[D|B]=n$, it is invertible. This gives that $V_1=0$. 

If $\beta=0$, we immediately obtian that $V_1=0$ by the boundary condition.

Now we assume that the matrix $(D,B)$ does not satisfy the Kalman rank condition. Then we know that either there exist $d_{j_1}$ and $d_{j_2}$ such that $d_{j_1}=d_{j_2}$, or there exists some $b_j=0$. We want to show the unique continuation property fails in both cases. One can refer to \cite{fattorini1966some} for more details. 

For the first case $b_j=0$, we know that
\begin{equation*}
    (\partial_t^2-d_j\Delta)u_j=0\text{ in }]0,T[\times\Omega,
\end{equation*}
by the conservation of energy, the solution $u_j$ cannot be zero at any time if the initial data is not zero.

For the second case, we consider the unique continuation property of the eigenfunctions as follows:
\begin{equation*}
    \left\{
    \begin{array}{l}
        -d_1\Delta \phi_1=\lambda^2\phi_1\text{ in }\Omega,\\
        \vdots\\
        -d_{j_1}\Delta \phi_{j_1}=\lambda^2\phi_{j_1}\text{ in }\Omega,\\
        -d_{j_2}\Delta \phi_{j_2}=\lambda^2\phi_{j_2}\text{ in }\Omega,\\
        \vdots\\
        -d_n\Delta \phi_n=\lambda^2\phi_n\text{ in }\Omega,\\
        \phi_j=0\quad\text{ on }\p\Omega,1\leq j\leq n,\\
    b_1\phi_1+\cdots+b_n\phi_n=0 \text{ in }\omega,
    \end{array}
    \right.
\end{equation*}
Since we have the relation $d_{j_1}=d_{j_2}$, we know that there exists a non-zero solution $(0,\cdots,0,\phi,-\frac{b_{j_1}}{b_{j_2}}\phi,0,\cdots,0)$, where $\phi$ is an eigenfunction for $-d_{j_1}\Delta$ of eigenvalue $\lambda^2$. Hence, we cannot obtain the exact controllability in this case.

To conclude, we have obtained that the Kalman rank condition is a sufficient and necessary condition for the exact controllabilty. 
\subsection{Two Generic Properties}\label{sec:Two Generic Properties}
If we define $\Delta_{K_1}=\Delta=\frac{d^2}{dx^2}$ and $n=2$, we have shown that not every smooth metric can give us a unique continuation result in dimension $1$ (see \Cref{sec:A counterexample}). Then we want to prove a generic property for the metrics which can give the unique continuation result in dimension $1$. 
We introduce the following space of smooth metrics to be sections of a bundle
\begin{equation*}
    \mathcal{M}=\{g\in C^{\infty}(\Omega,T^*\Omega\otimes T^*\Omega):g(x)(v_x,v_x)>0,\text{ for }0\neq v_x\in T_x\Omega\}.
\end{equation*}
Let $\Omega=]0,\pi[$.
\begin{prop}
In dimension $1$, suppose that we fix the Laplacian $\Delta=\frac{d^2}{dx^2}$ in $]0,\pi[$ with its spectrum $\sigma(\Delta)$. Then the set $\G_{uc}=\{g\in\mathcal{M}:\sigma(\Delta_g)\cap\sigma(\Delta)=\emptyset\}$ is residual in $\mathcal{M}$.
\end{prop}
\begin{proof}
First, we notice that any connected one dimensional Riemannian manifold is  diffeomorphic either to $\R$ or to $S^1$. We already know that $\sigma(\Delta)=\{k^2\}_{k\in\N}$. In our setting, we have $g=c(x)dx^2$. Then by change of variables, $y=\int_0^x\sqrt{c(s)}ds$. Then $\frac{d}{dy}=\frac{dx}{dy}\frac{d}{dx}=\frac{1}{\sqrt{c(x)}}\frac{d}{dx}$. Hence, we obtain $\frac{d^2}{dy^2}=\frac{1}{\sqrt{c(x)}}\frac{d}{dx}\frac{1}{\sqrt{c(x)}}\frac{d}{dx}=\Delta_g$. Define $L=\int_0^{\pi}\sqrt{c(s)}ds$. Hence, $\sigma(\Delta_g)=\sigma(\frac{d^2}{dy^2})=\{\frac{k^2\pi^2}{L^2}\}_{k\in\N}$.  If $\sigma(\Delta_g)\cap\sigma(\Delta)\neq\emptyset$, we obtain that for some $k$ and $l$, $L=\frac{k\pi}{l}\in\pi\Q$, i.e. $\int_0^{\pi}\sqrt{c(x)}\dif x\in \pi\Q$.
\end{proof}
\begin{coro}
Fix $\Delta=\frac{d^2}{dx^2}$, for every metric $g\in\G_{uc}$, the system \cref{eq:UC}
has a unique solution $u_1=u_2=0$.
\end{coro}
\begin{proof}
By the definition of $\G_{uc}$, we know $\sigma(\Delta_g)\cap\sigma(\Delta)=\emptyset$. Consider a solution $u_1$, $u_2$ of 
\begin{equation*}
\left\{
\begin{array}{cc}
     u''_1=-\lambda^2u_1, \\
     \Delta_gu_2=-\lambda^2u_2,\\
     u_1+u_2=0 \text{ in }]a,b[,\\
     u_1,u_2\in H^1_0(]0,\pi[).
\end{array}
\right.
\end{equation*} 
Now, assume that $u_1=0$. Then $u_2=0$ in $]a,b[$. Hence, by the unique continuation property for the eigenfunctions, we know that $u_2=0$. This means that the system has only trivial solution in this case. It is the same for $u_2=0$.

Assume that $u_1\neq0$ then $u_1\neq0$ in $]a,b[$(otherwise $u_1=0$ everywhere by the unique continuation property) and therefore $u_2\neq0$. Then $u_1$ and $u_2$ are both eigenfunctions. Hence $\lambda^2\in\sigma(\Delta_g)\cap\sigma(\Delta)=\emptyset$, which is a contradiction. So for every $g\in\G_{uc}$, the system has only the trivial solution $(0,0)$. 
\end{proof}
From now on and until the end of the section, we restrict to the $2$ dimensional case $d=2$. For any smooth metric $g$, we can define a Laplace-Beltrami operator $-\Delta_{g}$.
\begin{defi}
Define the map:
\begin{equation*}
    \mathcal{E}^{\alpha}: H^2(\Omega)\cap H^1_0(\Omega)\times \mathcal{M}\rightarrow H^{-1} 
\end{equation*}
by $\mathcal{E}^{\alpha}(u,g)=(\Delta_g+\alpha)u$.
\end{defi}
\begin{rem}
$-\Delta_g$ is a Fredholm operator of index $0$, and $\mathcal{E}_g^{\alpha}=\mathcal{E}^{\alpha}(\cdot,g)$ is also a Fredholm map of index $0$(see \cite{uhlenbeck1976generic}). Here $\alpha$ is just a parameter. In the later proof, we will let $\alpha$ take all possible values in the spectrum of the given Laplacian.
\end{rem}
From now on, we fix one metric $g_0$ and the associted operator $-\Delta_{g_0}$.
\begin{lem}\label{thm:condition}
For any $\lambda$ 
fixed and any element $f\in H^{-1}$, $\lambda\notin\sigma(\Delta_g)$ if and only if $f$ is a regular value (i.e. the tangential map at this point is surjective) of $\mathcal{E}^{\lambda}_g:H^2(\Omega)\cap H^1_0(\Omega)\rightarrow H^{-1} $.
\end{lem}
\begin{proof}
Let $\mathcal{E}^{\lambda}_g(u)=\mathcal{E}^{\lambda}(u,g)=f$. At this point $u$, the tangential map $D\mathcal{E}^{\lambda}_g:T_u(H^k(\Omega)\cap H^1_0(\Omega))\rightarrow H^{-1}(\Omega)$ is given by $D\mathcal{E}^{\lambda}_g(v)=(\Delta_g+\lambda)v$, since $\Delta_g+\lambda$ is a linear operator. $\lambda\notin\sigma(\Delta_g)$ is equivalent to that  $\Delta_g+\lambda$ is bijective, which means $f$ is a regular value of $\mathcal{E}^{\lambda}_g$.
\end{proof}
Our proof mainly rely on the following theorem:
\begin{thm}[Transversality theorem]\label{thm:transversality theorem}
Let $\varphi:H\times B\rightarrow E$ be a $C^k$ map, $H$, $B$, and $E$ Banach manifolds with $H$ and $E$ separable. If $f$ is a regular value of $\varphi$ and $\varphi_b=\varphi(\cdot,b)$ is a Fredholm map of index $<k$, then the set $\{b\in B:f\text{ is a regular value of }\varphi_b\}$ is residual in $B$.
\end{thm}
One can find a proof in \cite{abraham1963transversality}.
\begin{lem}
If $\lambda\in\sigma(\Delta_{g_0})$ is a regular value of $\mathcal{E}^{\lambda}$, then the set $\{g\in\mathcal{M}:\lambda\notin\sigma(\Delta_g)\}$ is residual in $\mathcal{M}$.
\end{lem}
\begin{proof}
Just apply \cref{thm:transversality theorem}, combining with \cref{thm:condition}.
\end{proof}

Now we have to check with the hypothesis, that is to verify that $\lambda\in\sigma(-\Delta_{g_0})$ is a regular value for $\mathcal{E}^{\lambda}$. In the following, we will use $D_1$ to denote the differential in the direction of $H^2(\Omega)\cap H^1_0(\Omega)$ and $D_2$ to denote the differential in the direction of $\mathcal{M}$.

Now let us check that the image of $D_2\mathcal{E}^{\lambda}$ is dense in dimension $2$. We will use the conformal variations of the metric $g$. Here we choose $r\in C^{\infty}_0(\Omega)$
\begin{equation}
\begin{aligned}
    D_2\mathcal{E}^{\lambda}(rg)&=\lim_{s\rightarrow0}\frac{(\Delta_{g+srg}-\Delta_g)u}{s}\\
                                &=\lim_{s\rightarrow0}\frac{1}{s}\left(\frac{1}{|(1+sr)g|^{\frac{1}{2}}}\partial_i|(1+sr)g|^{\frac{1}{2}}(1+sr)^{-1}g^{ij}\partial_ju-\Delta_gu\right)\\
                                &=\lim_{s\rightarrow0}\frac{1}{s}\left(\frac{2-2}{2}(1+sr)^{-2}\partial_irg^{ij}\partial_ju+\frac{1}{1+sr}\Delta_gu-\Delta_g u\right)\\
                                &=-r\Delta_gu
\end{aligned}
\end{equation}
Let us assume that $v$ is orthogonal to $D_2\mathcal{E}^{\lambda}(rg)$ for all $r$, then:
\begin{equation}\label{eq:orthogonal condition}
\begin{aligned}
    0&=\int_{\Omega}vD_2\mathcal{E}^{\lambda}(rg)d\mu_g\\
     &=\int_{\Omega}v(-r\Delta_gu)d\mu_g\\
     &=\int_{\Omega}r(\lambda u-\lambda)vd\mu_g.
\end{aligned}
\end{equation}
Since \cref{eq:orthogonal condition} holds for any $r\in C^{\infty}_0(\Omega)$ we obtain that:
\begin{equation}\label{eq:vanishing condition}
   (\lambda u-\lambda)v=0. 
\end{equation}

Now, we can check that $\lambda$ is a regular value of $\mathcal{E}^{\lambda}$.
\begin{lem}
In dimension $2$, $\lambda\in\sigma(\Delta_{g_0})$ is a regular value of $\mathcal{E}^{\lambda}$.
\end{lem}
\begin{proof}
Let $(u,g)$ satisfy $\mathcal{E}^{\lambda}(u,g)=(\Delta_g+\lambda)u=\lambda$, then at the point $(u,g)$, we have 
\begin{equation*}
    D\mathcal{E}^{\lambda}(v,h)=(\Delta_g+\lambda)v+D_2\mathcal{E}^{\lambda}(h).
\end{equation*}
Now we need to verify the surjectivity of this map. 
If $y\in Im(\Delta_g+\lambda)^{\perp}$, then $y$ is a weak solution of $(\Delta_g+\lambda)y=0$, and $y$ is smooth. Let us assume that $y$ is orthogonal to $D_2\mathcal{E}^{\lambda}(rg)$. Then according to \cref{eq:vanishing condition}, we obtain that:
\begin{equation*}
    (\lambda u-\lambda)y=0.
\end{equation*}
First, we claim that $u$ cannot be a constant. Assume that $u$ is a constant function, $\Delta_gu=0$ and $(\Delta_g+\lambda)u=\lambda$ gives that $u=1$. But this does not satisfy the boundary condition. Hence, $u$ cannot be a constant. In particular, $u\not\equiv1$. Now we obtain that $\lambda u-\lambda\not\equiv0$. If $\lambda u-\lambda\neq 0$ at $x_0$, there exists a open neighbourhood $N$ such that $\lambda u-\lambda\neq 0$ in $N$. Then $y\equiv0$ in $N$. Hence, we know that $y$ vanishes in a subdomain of $\Omega$. Then by the unique continuation property, we know $y=0$ in $\Omega$. This leads to the surjectivity of the map $D\mathcal{E}^{\lambda}$, which means that $\lambda\in\sigma(-\Delta_{g_0})$ is a regular value of $\mathcal{E}^{\lambda}$.
\end{proof}
Now we can deduce that the set $G^{\lambda}=\{g\in\mathcal{M}:\lambda\notin\sigma(\Delta_g)\}$ is residual in $\mathcal{M}$. 
\begin{prop}
In dimension $2$, suppose that we fix one metric $g_0$ and the associated  Laplacian $\Delta_{g_0}$ with its spectrum $\sigma(\Delta_{g_0})$. Then the set $\G_{uc}=\{g\in\mathcal{M}:\sigma(\Delta_g)\cap\sigma(\Delta_{g_0})=\emptyset\}$ is residual in $\mathcal{M}$.
\end{prop}
\begin{proof}
Define:
\begin{equation*}
    \G_{uc}=\cap_{\lambda\in\sigma(\Delta_{g_0})}G^{\lambda}.
\end{equation*}
$G$ is a intersection of countably many residual sets, so it is still residual in $\mathcal{M}$. And for any metric $g\in \G_{uc}$, $\sigma(\Delta_g)\cap\sigma(\Delta_{g_0})=\emptyset$. Assume that $\lambda_0\in\sigma(\Delta_g)\cap\sigma(\Delta_{g_0})$, which gives that $g\notin G^{\lambda_0}$. That contradicts to the fact that $g\in \G_{uc}=\cap_{\lambda\in\sigma(\Delta)}G^{\lambda}$. Hence, for fixed Laplacian $\Delta$ with its spectrum $\sigma(\Delta_{g_0})$, the set $\{g\in\mathcal{M}:\sigma(\Delta_g)\cap\sigma(\Delta_{g_0})=\emptyset\}$ is residual in $\mathcal{M}$.
\end{proof}
\begin{coro}
In dimension $2$, fix the canonical Laplace operator $\Delta$, for every metric $g\in\G_{uc}$, the system
\begin{equation*}
\left\{
\begin{array}{cc}
     \Delta u_1=-\lambda^2u_1,  \\
     \Delta_gu_2=-\lambda^2u_2,\\
     u_1+u_2=0 \text{ in }\omega\subset\Omega,\\
     u_1,u_2\in H^1_0(\Omega),
\end{array}
\right.
\end{equation*} 
has only trivial solution $u_1=u_2=0$.
\end{coro}

\section{Constant Coefficient Case with Multiple Control Functions}
\label{sec:Constant Coefficient Case with Multiple Control Functions}
In this section, we prove \cref{thm:mutli thm}. First we study the information given by the Kalman rank condition. Without loss of generality, we assume that the diagonal matrix $D$ has the form $D=\left(
\begin{array}{ccc}
     d_1Id_{n_1}& & \\
      &\ddots& \\
      & &d_sId_{n_s}
\end{array}
\right)$, where $\sum_{1\leq i\leq s}n_i=n$ and $d_i(1\leq i\leq s)$ are all distinct. And we can always rearrange the lines of the system \cref{eq:multi control system} to ensure that this property is verified:
\begin{equation*}
    \left\{
    \begin{aligned}
        (\partial_t^2-d_1\Delta)U_1&=B_1F\mathbf{1}_{]0,T[}(t)\mathbf{1}_{\omega}(x)\text{ in }]0,T[\times\Omega,\\
        &\vdots\\
        (\partial_t^2-d_s\Delta)U_s&=B_s F\mathbf{1}_{]0,T[}(t)\mathbf{1}_{\omega}(x)\text{ in }]0,T[\times\Omega,\\
    \end{aligned}
    \right.
\end{equation*}
for every $1\leq i\leq s$, where $U_i=\left(
\begin{array}{c}
     u^i_1 \\
     \vdots\\
     u^i_{n_i}
\end{array}
\right)$ and $B_i=\left(
            \begin{array}{ccc}
              b^i_{11} & \cdots & b^i_{1m} \\
              \vdots  & \ddots & \vdots \\
              b^i_{n_i1} & \cdots & b^i_{n_im}  \\
           \end{array}
          \right)$ is a matrix of size $n_i\times m$.
\begin{prop}\label{thm:multi Kalman condition}
$(D,B)$ satisfies the Kalman rank condition if and only if \\$rank(B_i)=n_i\leq m$.
\end{prop}
\begin{rem}
If $m=1$, we know that $rank(B_i)=n_i\leq1$. Thus, we obtain $n_i=1$ and $B_i=b_i\neq0$. This implies that every entry of control matrix $B$ is nonzero and all speeds $d_i$ are distinct. We recover the result of Remark 1.1 in \cite{ammar2009kalman}. If $m\geq2$, we can allow some block $d_i\id_{n_i}$ is of size $n_i\times n_i$, with $n_i\geq2$. For example, take 
$D=diag(1,1,2)$ and $B=\left(
            \begin{array}{cc}
              1& 0  \\
              0 &1  \\
              1 & 0  \\
           \end{array}
          \right)$. Then we obtain $[D|B]=\left(
            \begin{array}{cccccc}
              1& 0&1&0&1&0  \\
              0 &1 &0&1&0&1 \\
              4 & 0&2&0&1&0 \\
           \end{array}
          \right)$. Hence, we know that $rank[D|B]=3$ which means that the matrix $[D|B]$ is of full rank.
\end{rem}
The proof of \cref{thm:multi Kalman condition} is given in the Appendix.

Now we can prove \cref{thm:mutli thm}. 
\begin{proof}[Proof of \cref{thm:mutli thm}]
We follow the same procedure. Applying Hilbert uniqueness method, we can estabish the observability inequality:
\begin{equation}\label{eq:multi ob inequality}
    ||V(0)||^2_{(L^2\times H^{-1})^n}\leq C\int_0^T\int_{\omega}|B^*V|^2dxdt,
\end{equation}
where $B^*$ is the adjoint form of the matrix $B$, and $V= (V_1,\cdots,V_s)^t\in\R^{n_1}\times\cdots\times\R^{n_s}=\R^n$. Then we can estabilsh a similar weak observability inequality:
\begin{equation}\label{eq:multi weak ob}
||V(0)||^2_{(L^2\times H^{-1})^n}\leq C\int_0^T\int_{\omega}|B^*V|^2dxdt+C||V(0)||^2_{(H^{-1}\times H^{-2})^n} .
\end{equation}
Then argue by contradiction. Suppose that the weak observability inequality is false, then there exists a sequence $(V^k(0))_{k\in\N}$ such that 
\begin{equation}\label{eq:multi-norm-1}
    ||V^k(0)||^2_{(L^2\times H^{-1})^n}=1,
\end{equation}
\begin{equation}
    \int_0^T\int_{\omega}|B^*V^k|^2dxdt\rightarrow0,
\end{equation}
\begin{equation}
    ||V^k(0)||^2_{(H^{-1}\times H^{-2})^n}\rightarrow0   .
\end{equation}
Hence, there are $s$ microlocal defect measures $(\mu_i)_{i=1}^s$ corresponding to $V_i$. 
\begin{equation}
\begin{aligned}
    \int_0^T\int_{\omega}|B^*V^k|^2dxdt=\int_0^T\int_{\omega}|\sum_{i=1}^sB_i^*V_i^k|^2dxdt.
\end{aligned}
\end{equation}
Since $\mu_i$ and $\mu_j$ are singular from each other, for $i\neq j$, we know by Cauchy-Schwarz inequality, 
\begin{equation}
    \sum_{i=1}^s\int_0^T\int_{\omega}|B_i^*V_i^k|^2dxdt\rightarrow0,
\end{equation}
which gives that $B_iB^*_i\mu_i|_{\omega\times]0,T[}=0$. Since $rank(B_iB^*_i)=rank(B_i)=n_i$, we know $B_iB^*_i$ is invertible. Hence we know $\mu_i|_{\omega\times]0,T[}=0$. The rest of the proof is similar to the single control case.

\end{proof}
\appendix
\section{Proof of \cref{thm:multi Kalman condition}}
\label{appendix}
\begin{proof}[Proof of \cref{thm:multi Kalman condition}]
First, we calculate the form of $[D|B]$:
\begin{equation*}
\begin{aligned}
    \left[D|B\right]=&[D^{n-1}B|\cdots|DB|B]\\
         =&\left[\begin{array}{ccc}
              d_1^{n-1}B_1 & \cdots & B_1 \\
              \vdots  & \ddots & \vdots \\
              d_s^{n-1}B_s & \cdots & B_s  \\
           \end{array}\right]
\end{aligned}
\end{equation*}
Now we define $r_i=rank(B_i)$. Thus, for each $i$, we can find invertible matrices $P_i$ of size $n_i\times n_i$ and $Q_i$ of size $m\times m$ such that $P_iB_iQ_i=\left(\begin{array}{cc}
              \id_{r_i}  & 0 \\
              0& 0\\
           \end{array}\right)\overset{def}{=}E_i$. Then define $P=diag(P_1,\cdots,P_s)$ and $Q=diag(Q_1,\cdots, Q_s)$. We know that $P$ and $Q$ are invertible. Hence, we obtain $rank[D|B]=rank(P[D|B]Q)$. Now we rewrite that
\begin{equation*}
\begin{aligned}
    P[D|B]Q=&\left[\begin{array}{ccc}
              d_1^{n-1}P_1B_1Q_1 & \cdots & P_1B_1Q_s \\
              \vdots  & \ddots & \vdots \\
              d_s^{n-1}P_sB_sQ_1 & \cdots & P_sB_sQ_s  \\
           \end{array}\right]\\
           &=\left[\begin{array}{ccc}
              d_1^{n-1}E_1 & \cdots & P_1B_1Q_s \\
              \vdots  & \ddots & \vdots \\
              d_s^{n-1}P_sB_sQ_1 & \cdots & E_s  
           \end{array}\right]
\end{aligned}
\end{equation*} 
Now, consider the general term $P_iB_iQ_j$:
\begin{equation*}
    P_iB_iQ_j=P_iB_iQ_iQ_i^{-1}Q_j=E_iQ_i^{-1}Q_j.
\end{equation*}
Hence,
\begin{equation*}
    P[D|B]Q=\left[\begin{array}{ccc}
              d_1^{n-1}E_1 & \cdots & E_1Q_1^{-1}Q_s \\
              \vdots  & \ddots & \vdots \\
              d_s^{n-1}E_sQ_s^{-1}Q_1 & \cdots & E_s  
           \end{array}\right]
\end{equation*}
Now we define the column transform $T_1$:
\begin{equation*}
    T_1=\left[\begin{array}{cccc}
              \id_{n_1} &-\frac{1}{d_1}Q_1^{-1}Q_2& \cdots & -\frac{1}{d^{n-1}_1}Q_1^{-1}Q_s \\
              0&\id_{n_2}&\cdots&0\\
              \vdots  &\vdots& \ddots & \vdots \\
              0 & 0&\cdots & \id_{n_s}  
           \end{array}\right]
\end{equation*}
It is easy to see that $T_1$ is invertible and $rank(P[D|B]Q)=rank(P[D|B]QT_1)$.
\begin{equation*}
\begin{aligned}
    &P[D|B]QT_1\\
    =&\left[\begin{array}{cccc}
              d_1^{n-1}E_1 &0& \cdots & 0 \\
              d_2^{n-1}E_2Q_2^{-1}Q_1&(\frac{d_2^{n-1}}{d_2}-\frac{d_2^{n-1}}{d_1})E_2&\cdots&(\frac{d_2^{n-1}}{d_2^{n-1}}-\frac{d_2^{n-1}}{d_1^{n-1}})E_2Q_2^{-1}Q_s\\
              \vdots  &\vdots& \ddots & \vdots \\
              d_s^{n-1}E_sQ_s^{-1}Q_1 & \cdots&\cdots & (\frac{d_s^{n-1}}{d_s^{n-1}}-\frac{d_s^{n-1}}{d_1^{n-1}})E_s  
           \end{array}\right].
\end{aligned}
\end{equation*}
Step by step, we can do the Gaussian elimination and find an invertible matrix $T$ such that:
\begin{equation*}
    P[D|B]QT=\left[\begin{array}{cccc}
              d_1^{n-1}E_1 &0& \cdots & 0 \\
              *&d_2^{n-1}(\frac{1}{d_2}-\frac{1}{d_1})E_2&\cdots&0\\
              \vdots  &\vdots& \ddots & \vdots \\
              * & *&\cdots & d_s^{n-1}\prod_{i=1}^{s-1}(\frac{1}{d_s}-\frac{1}{d_i})E_s  
           \end{array}\right].
\end{equation*}
Then $rank[D|B]=rank(P[D|B]Q)=rank(P[D|B]Q)=\sum_{i=1}^s r_i\leq \sum_{i=1}^s n_i$. Hence, $n=rank[D|B]=\sum_{i=1}^s r_i\leq\sum_{i=1}^s n_i=n$. This implies that $rank[D|B]=n\Longleftrightarrow \forall i, r_i=n_i$.
\end{proof}

\section*{Acknowledgments}
This work is finished under the guidance of Nicolas Burq and Pierre Lissy, to whom the author owes great gratitude. The author would also express gratitude to Hui Zhu for several interesting discussions and comments.

\bibliographystyle{plain}
\bibliography{bib}

\end{document}